\documentclass[11pt,oneside]{article}

\usepackage[utf8]{inputenc}
\usepackage[english]{babel}
\usepackage[nottoc]{tocbibind}

\usepackage{lmodern}

\usepackage{mathtools}
\usepackage{fancyhdr}
\usepackage{amsfonts}
\usepackage{amsmath}
\usepackage{amssymb}
\usepackage{mathrsfs}
\usepackage{cite}
\usepackage{enumitem}
\usepackage{import}
\usepackage{tabularx}
\usepackage{latexsym}
\usepackage{shortvrb}
\usepackage{graphicx}
\usepackage{makeidx}
\usepackage[all]{xy}
\usepackage{amsthm}
\usepackage{fdsymbol}
\usepackage{physics}

\usepackage{hyperref}
\hypersetup{
	colorlinks=true,
	linkcolor=blue,
	filecolor=magenta,      
	urlcolor=cyan,
}
\usepackage{geometry}

\usepackage{pgf,tikz,pgfplots}
\pgfplotsset{compat=1.17}
\usetikzlibrary{arrows}
\usetikzlibrary{matrix}

\newtheorem{thm}{Theorem}[section]
\newtheorem{lem}[thm]{Lemma}

\newtheorem{ex}[thm]{Example}

\theoremstyle{definition}
\newtheorem{definition}[thm]{Definition}
\newtheorem{obs}[thm]{Observation}

\newtheorem{notation}[thm]{Notation}
\newtheorem{prob}[thm]{Problem}

\newtheorem{note}[thm]{Note}

\DeclareMathOperator{\Cat}{\mathbf{Cat}}
\DeclareMathOperator{\Rel}{\mathbf{Rel}}
\DeclareMathOperator{\Set}{\mathbf{Set}}
\DeclareMathOperator{\Vect}{\mathbf{Vect}}
\DeclareMathOperator{\Span}{\mathbf{Span}}
\DeclareMathOperator{\Pres}{\mathbf{Pres}}
\DeclareMathOperator{\Coalg}{\mathbf{Coalg}}
\DeclareMathOperator{\lax}{\mathbf{Lax}}
\DeclareMathOperator{\oplax}{\mathbf{Oplax}}
\DeclareMathOperator{\ps}{\mathbf{Ps}}
\DeclareMathOperator{\PGM}{\mathbf{PGM}}
\DeclareMathOperator{\catb}{\mathcal{B}}
\DeclareMathOperator{\catc}{\mathcal{C}}
\DeclareMathOperator{\catd}{\mathcal{D}}
\DeclareMathOperator{\cate}{\mathcal{E}}
\DeclareMathOperator{\catG}{\mathcal{G}}
\DeclareMathOperator{\mthz}{\mathbb{Z}}

\newcommand\authormark[1]{\textsuperscript{#1}}

\usepackage{amsmath,amssymb}
\usepackage{comment}

\newcommand{\jonoteil}[1]{\ \todo[inline,color=blue!40,linecolor=blue!40!black,size=\normalsize]{#1}}
\newcommand{\rmnoteil}[1]{\ \todo[inline,color=green!40,linecolor=green!40!black,size=\normalsize]{#1}}

\newcommand{\rtarr}{\longrightarrow}
\newcommand{\ltarr}{\longleftarrow}
\newcommand{\monoto}{\hookrightarrow}
\newcommand{\epito}{\twoheadrightarrow}
\newcommand{\fto}{\xrightarrow}
\newcommand{\fot}{\xleftarrow}
\newcommand{\ot}{\leftarrow}
\newcommand{\impl}{\Rightarrow}
\newcommand{\toeq}{\stackrel{\simeq}{\to}}
\newcommand{\mh}{\mbox{-}} 
\newcommand{\cn}{\colon}
\newcommand{\id}{\mathrm{id}}
\newcommand{\Id}{\mathrm{Id}}
\newcommand{\ob}{\mathrm{ob}}
\newcommand{\e}{\mathrm{e}}
\newcommand{\al}{\alpha}
\newcommand{\be}{\beta}
\newcommand{\ga}{\gamma}
\newcommand{\de}{\delta}
\newcommand{\pa}{\partial}   
\newcommand{\epz}{\varepsilon}
\newcommand{\ph}{\phi}
\newcommand{\phy}{\varphi}
\newcommand{\phz}{\varphi}
\newcommand{\et}{\eta}
\newcommand{\io}{\iota}
\newcommand{\ka}{\kappa}
\newcommand{\la}{\lambda}
\newcommand{\tha}{\theta}
\newcommand{\thz}{\vartheta}
\newcommand{\si}{\sigma}
\newcommand{\ze}{\zeta}
\newcommand{\om}{\omega}
\newcommand{\Ga}{\Gamma}
\newcommand{\GA}{\Gamma}
\newcommand{\La}{\Lambda}
\newcommand{\LA}{\Lambda}
\newcommand{\De}{\Delta}
\newcommand{\DE}{\Delta}
\newcommand{\Si}{\Sigma}
\newcommand{\SI}{\Sigma}
\newcommand{\Th}{\Theta}
\newcommand{\Om}{\Omega}
\newcommand{\OM}{\Omega}
\newcommand{\Ups}{\Upsilon}
\newcommand{\UPS}{\Upsilon}
\newcommand{\cc}{\mathbb{C}}
\newcommand{\wt}{\widetilde}
\newcommand{\strictslice}[2]{#1\hspace{-2pt}\sslash_{#2}^{\textrm{str}}}
\newcommand{\strongslice}[2]{#1\hspace{-2pt}\sslash_{#2}}
\newcommand{\conslice}[2]{#1\hspace{-2pt}\sslash_{#2}^{\textrm{adj}}}
\newcommand{\opconslice}[2]{#1\hspace{-2pt}\sslash_{#2}^{\textrm{opadj}}}
\renewcommand{\cal}{\mathcal}
\newcommand{\Simp}{\textbf{SimpSet}}
\newcommand{\omegagpd}{\omega\textbf{Gpd}}
\newcommand{\Gpd}{\textbf{Gpd}}
\newcommand{\Grp}{\textbf{Grp}}
\newcommand{\commMon}{\textbf{commMon}}
\newcommand{\stB}{B^*}
\newcommand{\Hst}{H^\ast}
\newcommand{\decB}{(\stB,B)}
\newcommand{\bul}{\bullet}
\newcommand{\acts}{\curvearrowright}
\newcommand{\dcat}{\textbf{dCat}}
\newcommand{\bcat}{\textbf{bCat}}
\newcommand{\bcatast}{\bcat ^\ast}
\newcommand{\twocatast}{\textbf{2Cat}^\ast}
\newcommand{\gcat}{\textbf{gCat}}
\newcommand{\Bdecor}{\decB}
\newcommand{\EndCat}{\textbf{$\pi_2$Ind}_{\Bdecor}}
\newcommand{\opIndCat}{\textbf{$\pi_2$opInd}_{\Bdecor}}
\newcommand{\gcatB}{\textbf{gCat}$_{(B^*,B)}$}
\newcommand{\gcatBfirst}{\textbf{gCat}$_{(B^*,B)}^{\ell=1}$}
\newcommand{\Hor}{\textbf{Hor}}
\newcommand{\boldH}{\mathbb{H}}
\newcommand{\ovC}{\overline{C}}
\newcommand{\crossb}{B\rtimes_\Phi B^\ast}
\newcommand{\intGrot}{\int_{\stB}\Phi}
\newcommand{\opintGrot}{\int_{B^{*op}}\Phi}
\newcommand{\GcopB}{\intGrot\sqcup B_1}
\newcommand{\GcoppiB}{\intGrot\sqcup_\pi B_1}
\newcommand{\indFunct}{\Phi:\stB\to\commMon}
\newcommand{\doubmod}{\mathbb{M}\mbox{\textbf{od}}}
\newcommand{\doubprof}{\mathbb{P}\mbox{\textbf{rof}}}
\newcommand{\doubvn}{[W^\ast]}
\newcommand{\doubbord}{\mathbb{B}\mbox{\textbf{ord}}}
\newcommand{\doubadj}{\mathbb{A}\mbox{\textbf{dj}}}

\makeindex

\title{Internalizations of decorated bicategories via $\pi_2$-indexings.}
\author{Juan Orendain\thanks{Corresponding author. Contact: juan.orendain@case.edu. Mathematics, Applied Mathematics, and Statistics department, Case Western Reserve University, 10900 Euclid Ave., Cleveland, OH, USA.} and Jos\'e Rub\'en Maldonado-Herrera\thanks{Contact: rubasmh@matmor.unam.mx. Centro de Ciencias Matem\'aticas, UNAM, Antigua Carretera a Pátzcuaro 8701
Col. Ex Hacienda San José de la Huerta, Morelia, Michoacan, Mexico.}}

\begin{document}

\maketitle
\abstract{We treat the problem of lifting bicategories into double categories through categories of vertical morphisms. We consider structures on decorated 2-categories allowing us to formally implement arguments of sliding certain squares along vertical subdivisions in double categories. We call these structures $\pi_2$-indexings. We present a construction associating, to every $\pi_2$-indexing on a decorated 2-category, a length 1 double internalization.}

\section{Introduction}\label{sec:intro}
Symmetric monoidal structures on bicategories are, in some cases, more naturally expressed as horizontalizations of symmetric monoidal structures of double categories. In order to define a symmetric monoidal structure on a bicategory $B$, in the general sense of \cite{GordonPowerStreet}, it is sometimes convenient to 'lift' $B$ to a symmetric monoidal double category $C$, where the coherence equations for the symmetric monoidal structure are expressed in terms of vertical arrows of $C$. In the case in which the double category $C$ is a framed bicategory, the image of the symmetric monoidal structure on $C$, under the horizontal bicategory functor $H$, defines a symmetric monoidal bicategory structure on $B$ \cite{ShulmanFramed,ShulmanDerived}. Whether this technique can be applied to a given bicategory $B$ depends on the specifics of $B$. We are interested in the problem of establishing general criteria for when this procedure can be applied. 

We consider the following problem: Given a bicategory $B$, and a category $B^*$, such that the collections of objects of $B$ and $\stB$ are equal, we wish to construct interesting double categories $C$ having $B$ as horizontal bicategory, and having $\stB$ as category of objects. Succintly $H^*C=(\stB,B)$. We say that the pair $\decB$ is a decorated bicategory and $C$ is an internalization of $\decB$. The problem of finding interesting internalizations to a given decorated bicategory $\decB$ models the case in which one wishes to define a symmetric monoidal structure on the bicategory $B$, where coherence data would be more naturally expressed in terms of a collection of vertical arrows, $\stB$. We then wish to have knowledge of the possible extensions of $B$ to double categories $C$ having $\stB$ as category of objects, i.e. $H^\ast C=\decB$. With this knowledge we would, in principle, be able to decide whether we can choose, among the internalizations of $\decB$, a framed double category where the symmetric monoidal structure we are trying to define on $B$ is naturally defined. 

The problem of understanding internalizations of decorated bicategories has been considered in the series of papers \cite{Orendain1,Orendain2,Orendain3}, where the 2-category of double categories is fibered over a coreflective subcategory minimizing the internalization problem. This allows the definition of a numerical invariant, called the vertical length $\ell C$, associated to every double category $C$. Roughly, the number $\ell C$ measures the amount of work one would be expected to do to construct a generic square in $C$, from squares in $H^\ast C$. 1 is the minimum possible length of a double category, and the condition $\ell C=1$ roughly means that we can express squares of $C$ in terms of squares of $H^\ast C$ in a relatively straightforward way. Most framed bicategories in the literature, e.g. $\doubmod,\doubprof,\doubbord,\doubadj$ are of length 1. It is thus natural to conjecture that every framed bicategory is of length 1. This conjecture will be treated in a subsequent paper. 

In this paper we consider the problem of deciding whether a decorated bicategory $\decB$ admits internalizations of length 1. The condition of a double category $C$ being of length 1 depends on certain 2-dimensional compatibility conditions satisfied by squares in $H^\ast C$, see \ref{ss:lengths}. We study a type of structure allowing for the construction of internalizations satisfying this compatibility condition. We call the structure we study $\pi_2$-indexings. $\pi_2$-indexings are a type of indexing associated to a decorated bicategory $\decB$, relating the arrows of $\stB$ with 2-cells of a specific type in $B$. The definition of $\pi_2$-indexing appears in Definition \ref{def:opindexing}, where related structures, which we call $\pi_2$-opindexings also appear. The main goal of the paper is to prove the following result 

\begin{thm}\label{thm:main}
    Let $\decB$ be a decorated 2-category. For every $\pi_2$-indexing, or $\pi_2$-opindexing $\Phi$ on $\decB$, there exists an internalization $\crossb$ of $\decB$ such that $\crossb$ is of length 1.
\end{thm}

\noindent We present examples and conjectures related to the construction in Theorem \ref{thm:main}.

We now present the contents of this paper. In Section \ref{s:Prel} we briefly review the notions of internalization, globularly generated double category, and length. In Section \ref{s:mainthm} we present our definition of $\pi_2$-indexing, and state our main theorem. We review some of the intuitive ideas behind the definition of $\pi_2$-indexings. In Section \ref{sec:internalizationsendindex} we present the proof of our main theorem. In Section \ref{s:examples} we present examples relevant to our construction. Finally, in Section \ref{s:outlook} we present open problems and possible directions related to the results presented in this paper.


\section{Preliminaries}\label{s:Prel}

\subsection{Double categories}\label{ss:Doubelcats}
\noindent We will write $C_0,C_1$ for the category of objects and the category of morphisms of a double category $C$, we will write $L,R$ for the left and right frame functors of $C$, and we will write $U$ for the unit functor of $C$. Following \cite{ShulmanDerived} we will write $\boxminus,\boxvert$ for the vertical and horizontal composition operations of $C$ respectively. We will write \textbf{dCat} for the 2-category of 'weak double categories', pseudo-double functors, and horizontal natural transformations \cite[Section 1.4]{GrandisPare}. We will represent squares in double categories as diagrams of the form:
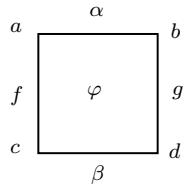
\begin{center}
    
\tikzset{every picture/.style={line width=0.75pt}} 

\begin{tikzpicture}[x=0.75pt,y=0.75pt,yscale=-1,xscale=1]

\draw   (240.73,120.3) -- (300.9,120.3) -- (300.9,180.47) -- (240.73,180.47) -- cycle ;

\draw (263.92,144.91) node [anchor=north west][inner sep=0.75pt]  [font=\scriptsize]  {$\varphi $};
\draw (225.07,145.2) node [anchor=north west][inner sep=0.75pt]  [font=\scriptsize]  {$f$};
\draw (306.64,144.91) node [anchor=north west][inner sep=0.75pt]  [font=\scriptsize]  {$g$};
\draw (265,105) node [anchor=north west][inner sep=0.75pt]  [font=\scriptsize]  {$\alpha $};
\draw (265.92,185.26) node [anchor=north west][inner sep=0.75pt]  [font=\scriptsize]  {$\beta $};
\draw (225.07,113.2) node [anchor=north west][inner sep=0.75pt]  [font=\scriptsize]  {$a$};
\draw (225.07,174.2) node [anchor=north west][inner sep=0.75pt]  [font=\scriptsize]  {$c$};
\draw (305.07,174.2) node [anchor=north west][inner sep=0.75pt]  [font=\scriptsize]  {$d$};
\draw (306.07,113.2) node [anchor=north west][inner sep=0.75pt]  [font=\scriptsize]  {$b$};

\end{tikzpicture}

\end{center}

\noindent where squares are read from left to right and from top to bottom, i.e. the edges $f,g$ are the left and right frame of the square $\varphi$, while the edges $\alpha,\beta$ are the vertical domain and codomain of $\varphi$ respectively. We will adopt the following convention: Vertical and horizontal identity arrows will always be colored red, e.g. in the following squares
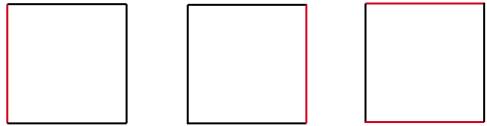
\begin{center}

\tikzset{every picture/.style={line width=0.75pt}} 

\begin{tikzpicture}[x=0.75pt,y=0.75pt,yscale=-1,xscale=1]

\draw    (260.07,139.97) -- (320.23,139.97) ;
\draw [color={rgb, 255:red, 208; green, 2; blue, 27 }  ,draw opacity=1 ]   (260.07,139.97) -- (260.07,200.13) ;
\draw    (260.07,200.13) -- (320.23,200.13) ;
\draw    (320.23,139.97) -- (320.23,200.13) ;
\draw    (351.07,140.3) -- (411.23,140.3) ;
\draw    (351.07,140.3) -- (351.07,200.47) ;
\draw    (350.73,200.13) -- (410.9,200.13) ;
\draw [color={rgb, 255:red, 208; green, 2; blue, 27 }  ,draw opacity=1 ]   (410.9,139.97) -- (410.9,200.13) ;
\draw [color={rgb, 255:red, 208; green, 2; blue, 27 }  ,draw opacity=1 ]   (440.73,139.63) -- (500.9,139.63) ;
\draw    (440.73,139.63) -- (440.73,199.8) ;
\draw [color={rgb, 255:red, 208; green, 2; blue, 27 }  ,draw opacity=1 ]   (440.4,199.47) -- (500.57,199.47) ;
\draw    (500.57,139.3) -- (500.57,199.47) ;

\end{tikzpicture}

\end{center}
the left frame of the square on the left is an identity, the right frame of the square in the middle is an identity and the domain and codomain of the rightmost square are both identities. We identify unit squares as being marked with the letter $U$:
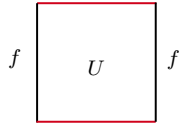
\begin{center}

\tikzset{every picture/.style={line width=0.75pt}} 

\begin{tikzpicture}[x=0.75pt,y=0.75pt,yscale=-1,xscale=1]

\draw [color={rgb, 255:red, 208; green, 2; blue, 27 }  ,draw opacity=1 ]   (300.07,139.3) -- (360.23,139.3) ;
\draw    (300.07,139.3) -- (300.07,199.47) ;
\draw [color={rgb, 255:red, 208; green, 2; blue, 27 }  ,draw opacity=1 ]   (299.73,199.13) -- (359.9,199.13) ;
\draw [color={rgb, 255:red, 0; green, 0; blue, 0 }  ,draw opacity=1 ]   (359.9,138.97) -- (359.9,199.13) ;

\draw (284,161.07) node [anchor=north west][inner sep=0.75pt]  [font=\scriptsize]  {$f$};
\draw (364,161.73) node [anchor=north west][inner sep=0.75pt]  [font=\scriptsize]  {$f$};
\draw (324,166.83) node [anchor=north west][inner sep=0.75pt]  [font=\scriptsize]  {$U$};

\end{tikzpicture}

\end{center}
\noindent An important class of squares is the collection squares of the form 
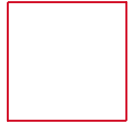
\begin{center}

\tikzset{every picture/.style={line width=0.75pt}} 

\begin{tikzpicture}[x=0.75pt,y=0.75pt,yscale=-1,xscale=1]

\draw [color={rgb, 255:red, 208; green, 2; blue, 27 }  ,draw opacity=1 ]   (320.07,159.3) -- (380.23,159.3) ;
\draw [color={rgb, 255:red, 208; green, 2; blue, 27 }  ,draw opacity=1 ]   (320.07,159.3) -- (320.07,219.47) ;
\draw [color={rgb, 255:red, 208; green, 2; blue, 27 }  ,draw opacity=1 ]   (319.73,219.13) -- (379.9,219.13) ;
\draw [color={rgb, 255:red, 208; green, 2; blue, 27 }  ,draw opacity=1 ]   (379.9,158.97) -- (379.9,219.13) ;

\end{tikzpicture}

\end{center}
\noindent We will write $\pi_2(C,a)$ for the set of all such squares with vertex $a$. Observe that by the Eckmann-Hilton argument, the horizontal composition and the vertical composition of such squares coincide and are commutative. The set $\pi_2(C,a)$ is thus a commutative monoid, with the following square as identity element:
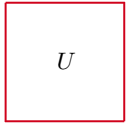
\begin{center}

\tikzset{every picture/.style={line width=0.75pt}} 

\begin{tikzpicture}[x=0.75pt,y=0.75pt,yscale=-1,xscale=1]

\draw [color={rgb, 255:red, 208; green, 2; blue, 27 }  ,draw opacity=1 ]   (340.07,179.3) -- (400.23,179.3) ;
\draw [color={rgb, 255:red, 208; green, 2; blue, 27 }  ,draw opacity=1 ]   (340.07,179.3) -- (340.07,239.47) ;
\draw [color={rgb, 255:red, 208; green, 2; blue, 27 }  ,draw opacity=1 ]   (339.73,239.13) -- (399.9,239.13) ;
\draw [color={rgb, 255:red, 208; green, 2; blue, 27 }  ,draw opacity=1 ]   (399.9,178.97) -- (399.9,239.13) ;

\draw (364.33,203.07) node [anchor=north west][inner sep=0.75pt]  [font=\footnotesize]  {$U$};

\end{tikzpicture}

\end{center}
\noindent Squares of the form:
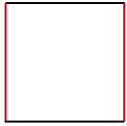
\begin{center}

\tikzset{every picture/.style={line width=0.75pt}} 

\begin{tikzpicture}[x=0.75pt,y=0.75pt,yscale=-1,xscale=1]

\draw [color={rgb, 255:red, 0; green, 0; blue, 0 }  ,draw opacity=1 ]   (300.07,119.3) -- (360.23,119.3) ;
\draw [color={rgb, 255:red, 208; green, 2; blue, 27 }  ,draw opacity=1 ]   (300.07,119.3) -- (300.07,179.47) ;
\draw [color={rgb, 255:red, 0; green, 0; blue, 0 }  ,draw opacity=1 ]   (299.73,179.13) -- (359.9,179.13) ;
\draw [color={rgb, 255:red, 208; green, 2; blue, 27 }  ,draw opacity=1 ]   (359.9,118.97) -- (359.9,179.13) ;

\end{tikzpicture}

\end{center}
\noindent are called globular. The collection of globular squares of a double category $C$ forms a bicategory $HC$, the horizontal bicategory of $C$. The horizontal bicategory construction assembles into a 2-functor $H:\dcat\to\bcat$, the horizontalization 2-functor. Thus defined $H$ admits a left adjoint $\boldH:\bcat\to\dcat$ associating to every bicategory $B$ the vertically trivial double category of $B$. Thus defined $\boldH$ is an embedding, and we identify every bicategory $B$ with its image $\boldH B$ under $\boldH$. Under this identification, bicategories are precisely the double categories whose squares are all globular.

\subsection{Internalization}\label{ss:Int}
\noindent A decoration of a bicategory $B$ is a category $B^\ast$ having the same collection of objects as $B$. If $B^\ast$ is a decoration of $B$ we say that $\decB$ is a \textbf{decorated bicategory}. The pair $(C_0,HC)$ is a decorated bicategory for every double category $C$. We write $H^*C$ for this decorated bicategory, and we call it the decorated horizontalization of $C$. We will write \textbf{bCat}$^*$ for the category of decorated bicategories and decorated 2-functors, see \cite{Orendain1}. We are interested in the following problem:

\begin{prob}\label{prob:Internal}
Let $\decB$ be a decorated bicategory. Find double categories $C$ satisfying the equation $H^*C=\decB$.
\end{prob}

\noindent We understand problem \ref{prob:Internal} as the problem of lifting a bicategory $B$ to a double category through an orthogonal direction, provided by $\stB$. A solution to Problem \ref{prob:Internal} for a decorated bicategory $\decB$ will be called an \textbf{internalization} of $B$. Problem \ref{prob:Internal} can be interpreted as follows: Suppose we are given a collection of globular squares:
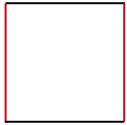
\begin{center}

\tikzset{every picture/.style={line width=0.75pt}} 

\begin{tikzpicture}[x=0.75pt,y=0.75pt,yscale=-1,xscale=1]

\draw [color={rgb, 255:red, 0; green, 0; blue, 0 }  ,draw opacity=1 ]   (300.07,119.3) -- (360.23,119.3) ;
\draw [color={rgb, 255:red, 208; green, 2; blue, 27 }  ,draw opacity=1 ]   (300.07,119.3) -- (300.07,179.47) ;
\draw [color={rgb, 255:red, 0; green, 0; blue, 0 }  ,draw opacity=1 ]   (299.73,179.13) -- (359.9,179.13) ;
\draw [color={rgb, 255:red, 208; green, 2; blue, 27 }  ,draw opacity=1 ]   (359.9,118.97) -- (359.9,179.13) ;

\end{tikzpicture}
\end{center}

\noindent forming a bicategory, together with a collection of vertical arrows of the form:
\begin{center}

\tikzset{every picture/.style={line width=0.75pt}} 

\begin{tikzpicture}[x=0.75pt,y=0.75pt,yscale=-1,xscale=1]

\draw [color={rgb, 255:red, 0; green, 0; blue, 0 }  ,draw opacity=1 ]   (320.07,139.3) -- (320.07,197.47) ;
\draw [shift={(320.07,199.47)}, rotate = 270] [color={rgb, 255:red, 0; green, 0; blue, 0 }  ,draw opacity=1 ][line width=0.75]    (6.56,-2.94) .. controls (4.17,-1.38) and (1.99,-0.4) .. (0,0) .. controls (1.99,0.4) and (4.17,1.38) .. (6.56,2.94)   ;

\end{tikzpicture}

\end{center}
\noindent forming a category, satisfying the condition that the collection of vertices of both sets of diagrams coincide. With this data we can form hollow squares of the form:

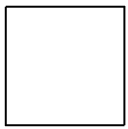
\begin{center}

\tikzset{every picture/.style={line width=0.75pt}} 

\begin{tikzpicture}[x=0.75pt,y=0.75pt,yscale=-1,xscale=1]

\draw [color={rgb, 255:red, 0; green, 0; blue, 0 }  ,draw opacity=1 ]   (320.07,139.3) -- (380.23,139.3) ;
\draw [color={rgb, 255:red, 0; green, 0; blue, 0 }  ,draw opacity=1 ]   (320.07,139.3) -- (320.07,199.47) ;
\draw [color={rgb, 255:red, 0; green, 0; blue, 0 }  ,draw opacity=1 ]   (319.73,199.13) -- (379.9,199.13) ;
\draw [color={rgb, 255:red, 0; green, 0; blue, 0 }  ,draw opacity=1 ]   (379.9,138.97) -- (379.9,199.13) ;

\end{tikzpicture}

\end{center}
\noindent formed by the edges of the diagrams we are provided with. Problem \ref{prob:Internal} asks about ways to fill these hollow squares \textit{equivariantly} with respect to the globular diagrams in our set of initial conditions. That is, Problem \ref{prob:Internal} asks for the existence of systems of solid squares of the form:

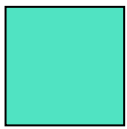
\begin{center}

\tikzset{every picture/.style={line width=0.75pt}} 

\begin{tikzpicture}[x=0.75pt,y=0.75pt,yscale=-1,xscale=1]

\draw  [fill={rgb, 255:red, 80; green, 227; blue, 194 }  ,fill opacity=0.57 ] (220.87,130.9) -- (280.7,130.9) -- (280.7,190.73) -- (220.87,190.73) -- cycle ;

\end{tikzpicture}

\end{center}

\noindent forming a double category such that the only solid squares of the form:

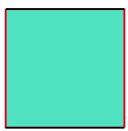
\begin{center}

\tikzset{every picture/.style={line width=0.75pt}} 

\begin{tikzpicture}[x=0.75pt,y=0.75pt,yscale=-1,xscale=1]

\draw  [color={rgb, 255:red, 255; green, 255; blue, 255 }  ,draw opacity=1 ][fill={rgb, 255:red, 80; green, 227; blue, 194 }  ,fill opacity=0.57 ] (240.87,150.9) -- (300.7,150.9) -- (300.7,210.73) -- (240.87,210.73) -- cycle ;
\draw [color={rgb, 255:red, 208; green, 2; blue, 27 }  ,draw opacity=1 ]   (240.87,150.9) -- (240.87,211.07) ;
\draw [color={rgb, 255:red, 208; green, 2; blue, 27 }  ,draw opacity=1 ]   (300.7,150.57) -- (300.7,210.73) ;
\draw [color={rgb, 255:red, 0; green, 0; blue, 0 }  ,draw opacity=1 ]   (240.87,210.73) -- (300.7,210.73) ;
\draw [color={rgb, 255:red, 0; green, 0; blue, 0 }  ,draw opacity=1 ]   (240.87,150.9) -- (300.7,150.9) ;

\end{tikzpicture}

\end{center}

\noindent are the globular diagrams provided as set of initial conditions. We regard the decorated horizontalization condition of Problem \ref{prob:Internal} as a formalization of the equivariance condition on the above squares. 

Constructions of this sort appear in different parts of the theory of double categories. Notably the Ehresmann double category of quintets \cite{Ehr3}, the double category of adjoint pairs \cite{Palmquist}, and the double categories of spans and cospans constructions all follow the pattern described above. The double category of quintets has a given 2-category and the corresponding category of 1-cells as set of initial data, the double category of adjoints has a given 2-category together with adjoint pairs of 1-cells as set of initial data, and the double category of spans/cospans has the bicategory of spans/cospans of a category with pushouts/pullbacks and the arrows of this category as globular and vertical sets of initial data. In all cases solid squares are carefully chosen so as to encode different aspects of the globular theory.

\subsection{Globularly generated double categories}\label{ss:ggdoublecats}
\noindent Globularly generated double categories were introduced in \cite{Orendain1} as minimal solutions to Problem \ref{prob:Internal}. A double category $C$ is \textbf{globularly generated} if $C$ is generated, as a double category, by its collection of globular squares. Pictorially a double category $C$ is globularly generated if every square of $C$ can be written as vertical and horizontal compositions of squares of the form:
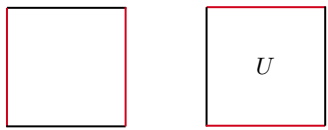
\begin{center}

\tikzset{every picture/.style={line width=0.75pt}} 

\begin{tikzpicture}[x=0.75pt,y=0.75pt,yscale=-1,xscale=1]

\draw [color={rgb, 255:red, 0; green, 0; blue, 0 }  ,draw opacity=1 ]   (230.07,140.3) -- (290.23,140.3) ;
\draw [color={rgb, 255:red, 208; green, 2; blue, 27 }  ,draw opacity=1 ]   (230.07,140.3) -- (230.07,200.47) ;
\draw [color={rgb, 255:red, 0; green, 0; blue, 0 }  ,draw opacity=1 ]   (229.73,200.13) -- (289.9,200.13) ;
\draw [color={rgb, 255:red, 208; green, 2; blue, 27 }  ,draw opacity=1 ]   (289.9,139.97) -- (289.9,200.13) ;
\draw [color={rgb, 255:red, 208; green, 2; blue, 27 }  ,draw opacity=1 ]   (330.73,139.97) -- (390.9,139.97) ;
\draw [color={rgb, 255:red, 0; green, 0; blue, 0 }  ,draw opacity=1 ]   (330.73,139.97) -- (330.73,200.13) ;
\draw [color={rgb, 255:red, 208; green, 2; blue, 27 }  ,draw opacity=1 ]   (330.4,199.8) -- (390.57,199.8) ;
\draw [color={rgb, 255:red, 0; green, 0; blue, 0 }  ,draw opacity=1 ]   (390.57,139.63) -- (390.57,199.8) ;

\draw (354,164.07) node [anchor=north west][inner sep=0.75pt]  [font=\footnotesize]  {$U$};

\end{tikzpicture}
\end{center}
\noindent Given a double category $C$ we write $\gamma C$ for the sub-double category of $C$ generated by squares of the above form. We call $\gamma C$ the \textbf{globularly generated piece} of $C$. $\gamma C$ is globularly generated, satisfies the equation 

\[H^*C=H^*\gamma C\]

\noindent and is contained in every sub-double category $D$ of $C$ satisfying the equation $H^*C=H^*D$. Moreover, a double category $C$ is globularly generated if and only if $C$ does not contain proper sub-double categories satisfying the above equation. Globularly generated double categories are thus minimal with respect to $H^*$.

Let \textbf{gCat} denote the full sub-category of \textbf{dCat} generated by globularly generated double categories. Decorated horizontalization extends to a functor $H^*:\mbox{\textbf{dCat}}\to\mbox{\textbf{bCat}}^*$ and the globularly generated piece construction extends to a functor $\gamma:\mbox{\textbf{dCat}}\to\mbox{\textbf{gCat}}$. In \cite[Proposition 3.6]{Orendain1} it is proven that $\gamma$ is a coreflector of \textbf{gCat} in \textbf{dCat}. It is easily seen that this implies that $\gamma$ is a Grothendieck fibration. Moreover, $H^*$ is constant on $\gamma$-fibers. We present this through the following diagram:

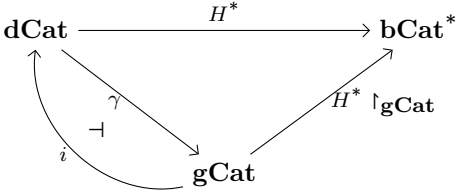
\begin{center}

\begin{tikzpicture}
\matrix(m)[matrix of math nodes, row sep=4em, column sep=4em,text height=1.5ex, text depth=0.25ex]
{\mbox{\textbf{dCat}}&&\mbox{\textbf{bCat}}^*\\
&\mbox{\textbf{gCat}}&\\};
\path[->,font=\scriptsize,>=angle 90]
(m-1-1) edge node [above]{$H^*$} (m-1-3)
        edge node [left]{$\gamma$} (m-2-2)
(m-2-2) edge node [right]{$H^*\restriction_{\mbox{\textbf{gCat}}}$}(m-1-3)
(m-2-2) edge [bend left=55] node [black,left]{$i$}(m-1-1)
(m-2-2) edge [white,bend left=30] node [black, fill=white]{$\dashv$}(m-1-1)

;
\end{tikzpicture}
\end{center}

\noindent where $i$ denotes the inclusion of \textbf{gCat} in \textbf{dCat}. The above diagram breaks Problem \ref{prob:Internal} into the problem of studying bases of $\gamma$ and then understanding the double categories in each fiber. We follow this strategy and thus study globularly generated double categories, i.e. bases with respect to $\gamma$. In \cite{Orendain3} it is proven that the above diagram can be completed to a diagram of the form:
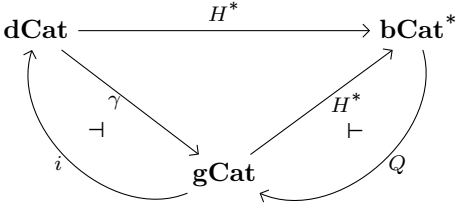
\begin{center}
    
\begin{center}

\begin{tikzpicture}
\matrix(m)[matrix of math nodes, row sep=4em, column sep=4em,text height=1.5ex, text depth=0.25ex]
{\mbox{\textbf{dCat}}&&\mbox{\textbf{bCat}}^*\\
&\mbox{\textbf{gCat}}&\\};
\path[->,font=\scriptsize,>=angle 90]
(m-1-1) edge node [above]{$H^*$} (m-1-3)
        edge node [left]{$\gamma$} (m-2-2)
(m-2-2) edge node [right]{$H^*$}(m-1-3)

(m-2-2) edge [bend left=65] node [black,left]{$i$}(m-1-1)
(m-2-2) edge [white,bend left=30] node [black, fill=white]{$\dashv$}(m-1-1)

(m-1-3) edge [bend left=65] node [black,right]{$Q$}(m-2-2)
(m-1-3) edge [white, bend left=30] node [black, fill=white]{$\vdash$}(m-2-2)

;
\end{tikzpicture}

\end{center}
\end{center}
\subsection{Length}\label{ss:lengths}
\noindent Globularly generated double categories admit a helpful combinatorial description provided in the form of a filtration of their categories of squares. Given a globularly generated double category $C$ we write $V^1_C$ for the category formed by vertical compositions of squares of the form:

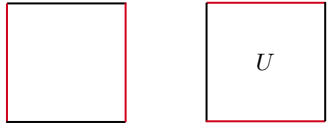
\begin{center}

\tikzset{every picture/.style={line width=0.75pt}} 

\begin{tikzpicture}[x=0.75pt,y=0.75pt,yscale=-1,xscale=1]

\draw [color={rgb, 255:red, 0; green, 0; blue, 0 }  ,draw opacity=1 ]   (230.07,140.3) -- (290.23,140.3) ;
\draw [color={rgb, 255:red, 208; green, 2; blue, 27 }  ,draw opacity=1 ]   (230.07,140.3) -- (230.07,200.47) ;
\draw [color={rgb, 255:red, 0; green, 0; blue, 0 }  ,draw opacity=1 ]   (229.73,200.13) -- (289.9,200.13) ;
\draw [color={rgb, 255:red, 208; green, 2; blue, 27 }  ,draw opacity=1 ]   (289.9,139.97) -- (289.9,200.13) ;
\draw [color={rgb, 255:red, 208; green, 2; blue, 27 }  ,draw opacity=1 ]   (330.73,139.97) -- (390.9,139.97) ;
\draw [color={rgb, 255:red, 0; green, 0; blue, 0 }  ,draw opacity=1 ]   (330.73,139.97) -- (330.73,200.13) ;
\draw [color={rgb, 255:red, 208; green, 2; blue, 27 }  ,draw opacity=1 ]   (330.4,199.8) -- (390.57,199.8) ;
\draw [color={rgb, 255:red, 0; green, 0; blue, 0 }  ,draw opacity=1 ]   (390.57,139.63) -- (390.57,199.8) ;

\draw (354,164.07) node [anchor=north west][inner sep=0.75pt]  [font=\footnotesize]  {$U$};

\end{tikzpicture}

\end{center}

\noindent and we denote by $H^1_C$ the (possibly weak) category formed by horizontal compositions of squares in $V^1_C$. Assuming we have defined $V^k_C$ and $H^k_C$ we make $V^{k+1}_C$ to be the category generated by morphisms in $H^k_C$ under vertical composition, and $H^{k+1}_C$ the (possibly weak) category generated by squares in $V^{k+1}_C$, under horizontal composition. The category of squares $C_1$ of $C$ satisfies the equation $C_1=\varinjlim V^k_C$. Then, we define the length $\ell C\in\mathbb{N}\cup\left\{\infty\right\}$ of any double category $C$ as the minimal $k$ such that the equation $\gamma C_1=V^k_{\gamma C}$ holds. Intuitively the vertical length of a double category $C$ measures the complexity of expressions of globularly generated squares in $C$. 

We further explain the vertical filtration construction pictorially: We regard the globular squares and horizontal identities of a double category $C$ as the simplest possible squares of $C$. We thus represent globular and horizontal identity squares diagrammatically as squares marked by 0, i.e. as:

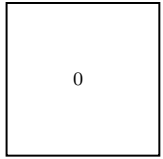
\begin{center}

\tikzset{every picture/.style={line width=0.75pt}} 

\begin{tikzpicture}[x=0.75pt,y=0.75pt,yscale=-1,xscale=1]

\draw   (292.77,151.6) -- (369.82,151.6) -- (369.82,228.65) -- (292.77,228.65) -- cycle ;

\draw (324.95,185.52) node [anchor=north west][inner sep=0.75pt]  [font=\scriptsize,xscale=0.9,yscale=0.9]  {$0$};

\end{tikzpicture}

\end{center}

\noindent Below we write $\mathbb{G}$ for the collection of 0-marked squares. Observe that the collection of 0-marked squares is closed under horizontal composition. Squares in $V^1_C$ are those squares in $C$ admitting a subdivision as vertical composition of 0-marked squares. Diagramatically:

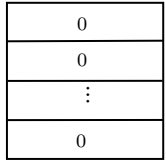
\begin{center}

\tikzset{every picture/.style={line width=0.75pt}} 

\begin{tikzpicture}[x=0.75pt,y=0.75pt,yscale=-1,xscale=1]

\draw   (290.34,171.28) -- (369.16,171.28) -- (369.16,250.09) -- (290.34,250.09) -- cycle ;
\draw    (290.41,190.71) -- (369.3,190.41) ;
\draw    (290.64,210.71) -- (368.86,210.66) ;
\draw    (290.19,230.49) -- (369.3,230.41) ;

\draw (324.45,177.54) node [anchor=north west][inner sep=0.75pt]  [font=\scriptsize,xscale=0.9,yscale=0.9]  {$0$};
\draw (324.34,195.76) node [anchor=north west][inner sep=0.75pt]  [font=\scriptsize,xscale=0.9,yscale=0.9]  {$0$};
\draw (323.9,236.43) node [anchor=north west][inner sep=0.75pt]  [font=\scriptsize,xscale=0.9,yscale=0.9]  {$0$};
\draw (333.41,211.46) node [anchor=north west][inner sep=0.75pt]  [font=\scriptsize,rotate=-90.3,xscale=0.9,yscale=0.9]  {$\cdots $};

\end{tikzpicture}

\end{center}

\noindent where we draw internal 0-marked squares as rectangles for convenience. If a square $\varphi$ as above is not a globular square or a horizontal identity, i.e. is not marked with 0, we mark $\varphi$ with 1. We represent 1-marked squares pictorially as:

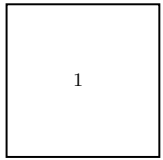
\begin{center}

\tikzset{every picture/.style={line width=0.75pt}} 

\begin{tikzpicture}[x=0.75pt,y=0.75pt,yscale=-1,xscale=1]

\draw   (292.77,169.6) -- (369.82,169.6) -- (369.82,246.65) -- (292.77,246.65) -- cycle ;

\draw (324.95,203.52) node [anchor=north west][inner sep=0.75pt]  [font=\scriptsize,xscale=0.9,yscale=0.9]  {$1$};

\end{tikzpicture}

\end{center}

\noindent Squares in $H^1_C$ are thus those squares in $C$ that admit a subdivision as horizontal composition of squares marked with $i\leq 1$. Given two horizontally composable squares $\varphi,\psi$ in $V^1_C$ we might be able to find compatible vertical subdivisions of $\varphi$ and $\psi$ in 0-marked squares, i.e. we might be able to represent the horizontal composition of $\varphi$ and $\psi$ as:

\begin{center}

\tikzset{every picture/.style={line width=0.75pt}} 

\begin{tikzpicture}[x=0.75pt,y=0.75pt,yscale=-1,xscale=1]

\draw   (251.59,141) -- (330.41,141) -- (330.41,219.82) -- (251.59,219.82) -- cycle ;
\draw    (251.66,160.43) -- (330.55,160.14) ;
\draw    (251.89,180.43) -- (330.11,180.38) ;
\draw    (251.44,200.21) -- (330.55,200.14) ;
\draw   (331.09,141) -- (409.91,141) -- (409.91,219.82) -- (331.09,219.82) -- cycle ;
\draw    (331.16,160.43) -- (410.05,160.14) ;
\draw    (331.39,180.43) -- (409.61,180.38) ;
\draw    (330.94,200.21) -- (410.05,200.14) ;

\draw (285.7,147.26) node [anchor=north west][inner sep=0.75pt]  [font=\scriptsize,xscale=0.9,yscale=0.9]  {$0$};
\draw (285.59,165.48) node [anchor=north west][inner sep=0.75pt]  [font=\scriptsize,xscale=0.9,yscale=0.9]  {$0$};
\draw (285.15,206.15) node [anchor=north west][inner sep=0.75pt]  [font=\scriptsize,xscale=0.9,yscale=0.9]  {$0$};
\draw (294.66,181.18) node [anchor=north west][inner sep=0.75pt]  [font=\scriptsize,rotate=-90.3,xscale=0.9,yscale=0.9]  {$\cdots $};
\draw (365.2,147.26) node [anchor=north west][inner sep=0.75pt]  [font=\scriptsize,xscale=0.9,yscale=0.9]  {$0$};
\draw (365.09,165.48) node [anchor=north west][inner sep=0.75pt]  [font=\scriptsize,xscale=0.9,yscale=0.9]  {$0$};
\draw (364.65,206.15) node [anchor=north west][inner sep=0.75pt]  [font=\scriptsize,xscale=0.9,yscale=0.9]  {$0$};
\draw (374.16,181.18) node [anchor=north west][inner sep=0.75pt]  [font=\scriptsize,rotate=-90.3,xscale=0.9,yscale=0.9]  {$\cdots $};

\end{tikzpicture}

\end{center}

\noindent where the internal 0-marked squares of the left and right outer squares match. In that case we can use the exchange identity to re-arrange the above composition into a vertical subdivision of 0-marked squares. Example \cite[Example 4.1]{Orendain2} shows that this is not always the case and that there might exist horizontally composable squares $\varphi,\psi$ such that any two vertical subdivisions into 0-squares look like:

\begin{center}

\tikzset{every picture/.style={line width=0.75pt}} 

\begin{tikzpicture}[x=0.75pt,y=0.75pt,yscale=-1,xscale=1]

\draw   (271.59,161) -- (350.41,161) -- (350.41,239.82) -- (271.59,239.82) -- cycle ;
\draw    (271.66,180.43) -- (350.55,180.14) ;
\draw    (271.89,197.58) -- (350.11,197.53) ;
\draw    (271.44,220.21) -- (350.55,220.14) ;
\draw   (351.09,161) -- (429.91,161) -- (429.91,239.82) -- (351.09,239.82) -- cycle ;
\draw    (351.16,190.15) -- (430.27,190.7) ;
\draw    (349.8,211.35) -- (429.85,211.23) ;

\draw (305.7,167.26) node [anchor=north west][inner sep=0.75pt]  [font=\scriptsize,xscale=0.9,yscale=0.9]  {$0$};
\draw (305.59,183.48) node [anchor=north west][inner sep=0.75pt]  [font=\scriptsize,xscale=0.9,yscale=0.9]  {$0$};
\draw (305.15,226.15) node [anchor=north west][inner sep=0.75pt]  [font=\scriptsize,xscale=0.9,yscale=0.9]  {$0$};
\draw (314.66,201.18) node [anchor=north west][inner sep=0.75pt]  [font=\scriptsize,rotate=-90.3,xscale=0.9,yscale=0.9]  {$\cdots $};
\draw (384.63,171.83) node [anchor=north west][inner sep=0.75pt]  [font=\scriptsize,xscale=0.9,yscale=0.9]  {$0$};
\draw (385.22,222.15) node [anchor=north west][inner sep=0.75pt]  [font=\scriptsize,xscale=0.9,yscale=0.9]  {$0$};
\draw (394.44,192.61) node [anchor=north west][inner sep=0.75pt]  [font=\scriptsize,rotate=-90.3,xscale=0.9,yscale=0.9]  {$\cdots $};

\end{tikzpicture}

\end{center}

\noindent i.e. where the internal 0-squares cannot be arranged to match horizontally. Such horizontal compositions are not 1-marked. We represent squares in $H^1_C$ as above, i.e. squares in $H^1_C\setminus V^1_C$ as squares marked with 1+1/2, i.e. as:

\begin{center}

\tikzset{every picture/.style={line width=0.75pt}} 

\begin{tikzpicture}[x=0.75pt,y=0.75pt,yscale=-1,xscale=1]

\draw   (292.77,179.6) -- (369.82,179.6) -- (369.82,256.65) -- (292.77,256.65) -- cycle ;

\draw (310.95,214.72) node [anchor=north west][inner sep=0.75pt]  [font=\scriptsize,xscale=0.9,yscale=0.9]  {$1+1/2$};

\end{tikzpicture}

\end{center}

\noindent $V^2_C$ is thus the category of squares admitting a vertical subdivision into squares marked with $\leq 1+1/2$. Inductively, given $k\geq 1$, $V^k_C$ is the category of squares admitting vertical subdivisions as:

\begin{center}

\tikzset{every picture/.style={line width=0.75pt}} 

\begin{tikzpicture}[x=0.75pt,y=0.75pt,yscale=-1,xscale=1]

\draw   (292.34,180.25) -- (371.16,180.25) -- (371.16,259.07) -- (292.34,259.07) -- cycle ;
\draw    (292.41,199.68) -- (371.3,199.39) ;
\draw    (292.64,219.68) -- (370.86,219.63) ;
\draw    (292.19,239.46) -- (371.3,239.39) ;

\draw (326.45,186.51) node [anchor=north west][inner sep=0.75pt]  [font=\scriptsize,xscale=0.9,yscale=0.9]  {$i_{1}$};
\draw (326.34,204.73) node [anchor=north west][inner sep=0.75pt]  [font=\scriptsize,xscale=0.9,yscale=0.9]  {$i_{2}$};
\draw (325.9,245.4) node [anchor=north west][inner sep=0.75pt]  [font=\scriptsize,xscale=0.9,yscale=0.9]  {$i_{s}$};
\draw (335.41,220.43) node [anchor=north west][inner sep=0.75pt]  [font=\scriptsize,rotate=-90.3,xscale=0.9,yscale=0.9]  {$\cdots $};

\end{tikzpicture}

\end{center}

\noindent where the $i_j$'s are all $\leq k-1/2$. Squares marked with $k$ are squares in $V^k_C$ not marked with $i<k$. $H^{k+1}_C$ is the (possibly weak) category of squares admitting a horizontal subdivision as:

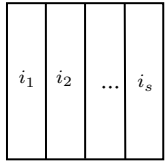
\begin{center}

\tikzset{every picture/.style={line width=0.75pt}} 

\begin{tikzpicture}[x=0.75pt,y=0.75pt,yscale=-1,xscale=1]

\draw   (369.66,181.75) -- (369.66,260.57) -- (290.84,260.57) -- (290.84,181.75) -- cycle ;
\draw    (350.22,181.82) -- (350.52,260.71) ;
\draw    (330.22,182.05) -- (330.27,260.27) ;
\draw    (310.45,181.6) -- (310.52,260.71) ;

\draw (295.45,214.51) node [anchor=north west][inner sep=0.75pt]  [font=\scriptsize,xscale=0.9,yscale=0.9]  {$i_{1}$};
\draw (314.34,214.23) node [anchor=north west][inner sep=0.75pt]  [font=\scriptsize,xscale=0.9,yscale=0.9]  {$i_{2}$};
\draw (355.4,216.4) node [anchor=north west][inner sep=0.75pt]  [font=\scriptsize,xscale=0.9,yscale=0.9]  {$i_{s}$};
\draw (349.23,226.06) node [anchor=north west][inner sep=0.75pt]  [font=\scriptsize,rotate=-180.3,xscale=0.9,yscale=0.9]  {$\cdots $};

\end{tikzpicture}

\end{center}

\noindent where the $i_j$'s are all $\leq k$. Squares marked with $k+1/2$ are those squares in $H^k_C$ such that no subdivision as above can be reduced as a vertical subdivision as $i$-squares with $i\leq k-1/2$. In \cite{Orendain2} it is shown that there exist globularly generated double categories such that squares marked with $k+1/2$ exist for every $k\geq 0$. The formula $C_1=\varinjlim V^k_C$ thus means that in a globularly generated double category $C$ every square admits a $\mathbb{N}+1/2\mathbb{N}$-marking as above. The length of a square $\varphi$ marked by $x\in \mathbb{N}+1/2\mathbb{N}$ is $\lceil x\rceil$ and the length $\ell C$ is the maximum of lengths of squares in $C$. The above pictorial representation is only meant to serve as intuition for the vertical filtration construction and we will not use it for the reminder of the paper.

\subsection{Assumptions}\label{ss:Assumptions}

\noindent We will assume, for the rest of the paper, that all double categories are strict, i.e. we will assume that horizontal composition is associative on the nose, and that unit squares act as units with respect to horizontal composition strictly. We will thus only consider decorated 2-categories and not general decorated bicategories, with some exceptions in subsection \ref{ss:endindexings}. We will denote by $\twocatast$ the subcategory of $\bcatast$ generated by decorated 2-categories. It is not difficult to see how the results of this paper could be extended to the general setting of decorated bicategories. This will be pursued in subsequent work.

\section{$\pi_2$-indexings and canonical decompositions}\label{s:mainthm}

\noindent Theorem \ref{thm:main} associates, to a certain type of structure defined on a decorated 2-category, an internalization of length 1. In this section we spell out what this structure is. We provide examples, and we then explain the main ideas behind how the proof of Theorem \ref{thm:main} works. The main argument for the proof of Theorem \ref{thm:main} is delayed until Section \ref{sec:internalizationsendindex}. 

\subsection{$\pi_2$-Indexings}\label{ss:endindexings}
\begin{definition}\label{def:opindexing}
Let $(B^*,B)\in\twocatast$. A \textbf{$\pi_2$-indexing} on $(B^*,B)$ is a functor
\[\indFunct\]
such that for every object $a$ in $\stB$, the monoid $\Phi(a)$ is the commutative monoid $\pi_2(B,a)$ of $a$, in $B$. A \textbf{$\pi_2$-opindexing} on $(B^*,B)$ is an indexing on $(B^{\ast op},B)$. We write $\pi_2\textbf{Ind}_{\decB}$ and $\pi_2\textbf{opInd}_{\decB}$ for the collection of $\pi_2$-indexings on $\decB$ and the collection of $\pi_2$-opindexings on $\decB$ respectively.
\end{definition}
\noindent $\pi_2$-indexings and $\pi_2$-opindexings are the types of structure we associate to decorated 2-categories in order to construct internalizations of length 1. We explain the ideas behind Definition \ref{def:opindexing} in \ref{ss:canonicalform}. We now present our main examples of $\pi_2$-indexings. 
\begin{ex}\label{ex:abeliangroups}
Let $A$ be an abelian group and let $G$ be a group. Consider the decorated 2-category $(\Omega G,2\Omega A)$, where $\Omega G$ denotes the delooping groupoid generated by $G$, i.e. $\Omega G$ is the groupoid with a single object $\bullet$, whose automorphism group is $G$, and where $2\Omega A$ is the 2-category with a single object $\bullet$, whose monoidal category of automorphisms is $\Omega A$. A $\pi_2$-indexing on the decorated bicategory $(\Omega G,2\Omega A)$ is a functor
\[\Phi:\Omega G\to \commMon\]
such that $\Phi(\bul)=A$. Given a morphism $g\in \Omega G$, i.e. given an element $g\in G$, the image $\Phi(g)$ of $g$, is an automorphism of $A$. $\Phi$ is thus precisely the data of a morphism $G\to Aut(A)$, i.e. $\Phi$ is an action by automorphisms $G\acts A$.
Analogously, a $\pi_2$-opindexing on $(\Omega G, 2\Omega A)$ is the data of a morphism $G^{op}\to Aut(A)$. 
\end{ex}

\begin{ex}\label{ex:tensorcats}
Let $\mathcal{C}$ be a tensor category over a field $\mathbb{K}$ and let $G$ be a group. Consider the decorated bicategory $(\Omega G, \Omega C)$, where $\Omega C$ is the bicategory with a single object $\bul$, whose category of endomorphisms is $C$. A $\pi_2$-indexing on the pair $(\Omega G, \Omega C)$ is a functor $\Phi: \Omega G \to \commMon$ such that $\Phi(\bul)= \text{End(1)} \cong \mathbb{K}$. Given an element $g\in G$, the image $\Phi(g)$ of $g$, is an automorphism of $\mathbb{K}$. Thus, $\Phi$ is precisely the data of a morphism $G\to \mathbb{K}^\times$, i.e. $\Phi$ is a character of $G$ on $\mathbb{K}$.
Analogously, a $\pi_2$-opindexing on $(\Omega G, \Omega \mathcal{C})$ is a character of $G^{op}$ on $\mathbb{K}$. 
\end{ex}

\begin{ex}
    Let $\textbf{Ring}^{\textbf{surj}}$ be the category of unital rings and surjective ring morphisms. Let $\textbf{Mod}$ be the bicategory of bimodules over unital rings. Consider the decorated bicategory $(\textbf{Ring}^{\textbf{surj}}, \textbf{Mod})$. Let $A$ be a ring. The monoid $\pi_2(\textbf{Mod}, A)$ is isomorphic to $Z(A)$. Now, if $f:A\to B$ is a surjective morphism, then $f(Z(A))\subset Z(B).$ The assignment $A\mapsto Z(A)$ and $f\mapsto f|_{Z(A)}$ assemble into a $\pi_2$-indexing on $(\textbf{Ring}^\textbf{surj}, \textbf{Mod})$: $$Z:\textbf{Ring}^\textbf{surj} \to \textbf{commMon}$$

    \noindent More generally, let $\textbf{Ring}^{\textbf{Z}}$ be the category of unital rings and ring morphisms that preserve centers. There is an  $\pi_2$-indexing on $(\textbf{Ring}^{\textbf{Z}}, \textbf{Mod})$: $$Z:\textbf{Ring}^{\textbf{Z}}\to \textbf{commMon}$$
\end{ex}\label{ex:spanscospansindexing}
\noindent The following are examples of decorated bicategories with a single $\pi_2$-indexing.
\begin{ex}
    Let $C$ be a category with enough pullbacks and let $\textbf{Span}(C)$ be the bicategory of spans over $C$. If $c$ is an object on $C$ there is only one endomorphism over the span
    
\begin{center}
\begin{tikzpicture}
\matrix(m)[matrix of math nodes, row sep=4em, column sep=2.5em,text height=1.5ex, text depth=0.25ex]
{c & c & c\\};
\path[->,font=\scriptsize,>=angle 90]
(m-1-2) edge node [above]{$id_c$} (m-1-1)
        edge node [above]{$id_c$} (m-1-3)
;
\end{tikzpicture}
\end{center}

\noindent i.e. $\pi_2(C,c)$ is trivial for every $c$ on $C$. Thus, there is only one $\pi_2$-indexing and only one $\pi_2$-opindexing on $(C,\textbf{Span}(C))$. Analogously, if $D$ is a category with enough pushouts and $\textbf{coSpan}(D)$ is the bicategory of cospans over $D$, then there is only one $\pi_2$-indexing and only one $\pi_2$-opindexing on $(D, \textbf{coSpan}(D)).$
\end{ex}

\begin{ex}\label{ex:bordismsindexings}
    Let $n\textbf{Man}$ denote the category of closed oriented $n$-manifolds and diffeomorphisms, and let $n\textbf{Cob}$ denote the bicategory of $(n+1)$-cobordisms with 2-cells being diffeomorphisms that fix tubular neighborhoods. If $X$ is a closed oriented n-manifold, then the identity of the cylinder $X\times [0,1]$ is the only element in $\pi_2(n\textbf{Cob},X)$. The decorated bicategory $(n\textbf{Man},n\textbf{Cob})$ thus admits only one indexing and one opindexing.
\end{ex}
\noindent The following is an example of a decorated 2-category not admitting $\pi_2$-indexings.
\begin{ex}\label{ex:noindexing}
Let $\stB$ be the groupoid generated by the free standing arrow
\begin{center}
\begin{tikzpicture}
\matrix(m)[matrix of math nodes, row sep=3em, column sep=3em,text height=1.5ex, text depth=0.25ex]
{\bullet\\
\bullet'\\};
\path[->,font=\scriptsize,>=angle 90]
(m-1-1) edge node [left]{$\alpha$} (m-2-1)

;
\end{tikzpicture}
\end{center}
\noindent i.e. let $B^\ast$ be the free living isomorphism, and let $B$ be the 2-category with two objects $\bullet,\bullet'$, no non-trivial 1-cells, and where $\pi_2(B,\bullet)$ and $\pi_2(B,\bullet')$ are not isomorphic commutative monoids. If there existed a $\Phi\in \EndCat$ or $\Phi\in\opIndCat$, then $\Phi(\alpha)$ would be an isomorphism between $\pi_2(\bullet,B)$ and $\pi_2(\bullet',B)$. 
    
\end{ex}

\subsection{Canonical decompositions}\label{ss:canonicalform}
\noindent The proof of Theorem \ref{thm:main} is presented in Section \ref{sec:internalizationsendindex}. In the reminder of this section we explain the main ideas behind considering $\pi_2$-indexings as the type of structure that would yield internalizations of length 1. 

All squares of length 1 in a globularly generated double category $C$ admit an expression as vertical composition of globular and unit squares. Moreover, in \cite[Lemma 4.6]{Orendain1} it is proven that every square 
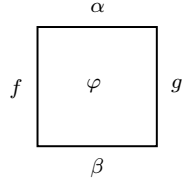
\begin{center}

\tikzset{every picture/.style={line width=0.75pt}} 

\begin{tikzpicture}[x=0.75pt,y=0.75pt,yscale=-1,xscale=1]

\draw   (260.73,140.3) -- (320.9,140.3) -- (320.9,200.47) -- (260.73,200.47) -- cycle ;

\draw (283.92,164.91) node [anchor=north west][inner sep=0.75pt]  [font=\scriptsize]  {$\varphi $};
\draw (245.07,165.2) node [anchor=north west][inner sep=0.75pt]  [font=\scriptsize]  {$f$};
\draw (326.64,164.91) node [anchor=north west][inner sep=0.75pt]  [font=\scriptsize]  {$g$};
\draw (285.92,126.77) node [anchor=north west][inner sep=0.75pt]  [font=\scriptsize]  {$\alpha $};
\draw (285.92,205.26) node [anchor=north west][inner sep=0.75pt]  [font=\scriptsize]  {$\beta $};

\end{tikzpicture}

\end{center}
of length 1 in a double category $C$ can be written as a vertical composition of the form
\begin{equation}\label{eq:decomposition}
\varphi_{\downarrow}\boxminus U_{n+1}\boxminus\psi_n\boxminus\cdots\boxminus\psi_1\boxminus U_1\boxminus\varphi_{\uparrow}
\end{equation}
\noindent where the top and bottom squares $\varphi_{\uparrow}$ and $\varphi_{\downarrow}$ of the above decomposition are globular squares of the form
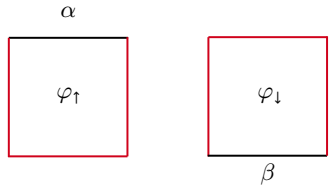
\begin{center}

\tikzset{every picture/.style={line width=0.75pt}} 

\begin{tikzpicture}[x=0.75pt,y=0.75pt,yscale=-1,xscale=1]

\draw [color={rgb, 255:red, 0; green, 0; blue, 0 }  ,draw opacity=1 ]   (250.07,160.3) -- (310.23,160.3) ;
\draw [color={rgb, 255:red, 208; green, 2; blue, 27 }  ,draw opacity=1 ]   (250.07,160.3) -- (250.07,220.47) ;
\draw [color={rgb, 255:red, 208; green, 2; blue, 27 }  ,draw opacity=1 ]   (249.73,220.13) -- (309.9,220.13) ;
\draw [color={rgb, 255:red, 208; green, 2; blue, 27 }  ,draw opacity=1 ]   (309.9,159.97) -- (309.9,220.13) ;
\draw [color={rgb, 255:red, 208; green, 2; blue, 27 }  ,draw opacity=1 ]   (350.73,159.97) -- (410.9,159.97) ;
\draw [color={rgb, 255:red, 208; green, 2; blue, 27 }  ,draw opacity=1 ]   (350.73,159.97) -- (350.73,220.13) ;
\draw [color={rgb, 255:red, 0; green, 0; blue, 0 }  ,draw opacity=1 ]   (350.4,219.8) -- (410.57,219.8) ;
\draw [color={rgb, 255:red, 208; green, 2; blue, 27 }  ,draw opacity=1 ]   (410.57,159.63) -- (410.57,219.8) ;

\draw (374,184.07) node [anchor=north west][inner sep=0.75pt]  [font=\footnotesize]  {$\varphi _{\downarrow }$};
\draw (272.57,183.92) node [anchor=north west][inner sep=0.75pt]  [font=\footnotesize]  {$\varphi _{\uparrow }$};
\draw (274.57,142.92) node [anchor=north west][inner sep=0.75pt]  [font=\footnotesize]  {$\alpha $};
\draw (375.57,222.4) node [anchor=north west][inner sep=0.75pt]  [font=\footnotesize]  {$\beta $};

\end{tikzpicture}

\end{center}
and where $U_1,....,U_{n+1}$ and $\psi_1,...,\psi_n$ are squares of the form
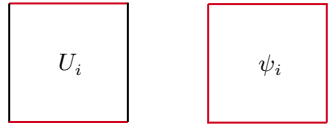
\begin{center}

\tikzset{every picture/.style={line width=0.75pt}} 

\begin{tikzpicture}[x=0.75pt,y=0.75pt,yscale=-1,xscale=1]

\draw [color={rgb, 255:red, 208; green, 2; blue, 27 }  ,draw opacity=1 ]   (250.07,160.3) -- (310.23,160.3) ;
\draw [color={rgb, 255:red, 208; green, 2; blue, 27 }  ,draw opacity=1 ]   (250.07,160.3) -- (250.07,220.47) ;
\draw [color={rgb, 255:red, 208; green, 2; blue, 27 }  ,draw opacity=1 ]   (249.73,220.13) -- (309.9,220.13) ;
\draw [color={rgb, 255:red, 208; green, 2; blue, 27 }  ,draw opacity=1 ]   (309.9,159.97) -- (309.9,220.13) ;
\draw [color={rgb, 255:red, 208; green, 2; blue, 27 }  ,draw opacity=1 ]   (149.73,159.97) -- (209.9,159.97) ;
\draw [color={rgb, 255:red, 0; green, 0; blue, 0 }  ,draw opacity=1 ]   (149.73,159.97) -- (149.73,220.13) ;
\draw [color={rgb, 255:red, 208; green, 2; blue, 27 }  ,draw opacity=1 ]   (149.4,219.8) -- (209.57,219.8) ;
\draw [color={rgb, 255:red, 0; green, 0; blue, 0 }  ,draw opacity=1 ]   (209.57,159.63) -- (209.57,219.8) ;

\draw (173,184.07) node [anchor=north west][inner sep=0.75pt]  [font=\footnotesize]  {$U_{i}$};
\draw (274,184.07) node [anchor=north west][inner sep=0.75pt]  [font=\footnotesize]  {$\psi _{i}$};

\end{tikzpicture}

\end{center}
i.e. where the $U_i$'s in the middle of the decomposition are unit squares, and the $\psi_i$'s are squares in $\pi_2(C,a_i)$ for some $a_i$. The above decomposition can be summarized as
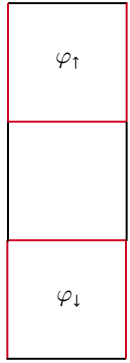
\begin{center}

\tikzset{every picture/.style={line width=0.75pt}} 

\begin{tikzpicture}[x=0.75pt,y=0.75pt,yscale=-1,xscale=1]

\draw [color={rgb, 255:red, 0; green, 0; blue, 0 }  ,draw opacity=1 ]   (191.07,60.3) -- (251.23,60.3) ;
\draw [color={rgb, 255:red, 208; green, 2; blue, 27 }  ,draw opacity=1 ]   (191.07,60.3) -- (191.07,120.47) ;
\draw [color={rgb, 255:red, 208; green, 2; blue, 27 }  ,draw opacity=1 ]   (190.73,120.13) -- (250.9,120.13) ;
\draw [color={rgb, 255:red, 208; green, 2; blue, 27 }  ,draw opacity=1 ]   (250.9,59.97) -- (250.9,120.13) ;
\draw [color={rgb, 255:red, 208; green, 2; blue, 27 }  ,draw opacity=1 ]   (190.73,179.97) -- (250.9,179.97) ;
\draw [color={rgb, 255:red, 208; green, 2; blue, 27 }  ,draw opacity=1 ]   (190.73,179.97) -- (190.73,240.13) ;
\draw [color={rgb, 255:red, 0; green, 0; blue, 0 }  ,draw opacity=1 ]   (190.4,239.8) -- (250.57,239.8) ;
\draw [color={rgb, 255:red, 208; green, 2; blue, 27 }  ,draw opacity=1 ]   (250.57,179.63) -- (250.57,239.8) ;
\draw [color={rgb, 255:red, 0; green, 0; blue, 0 }  ,draw opacity=1 ]   (191.07,120.47) -- (191.07,180.63) ;
\draw [color={rgb, 255:red, 0; green, 0; blue, 0 }  ,draw opacity=1 ]   (250.9,120.13) -- (250.9,180.3) ;

\draw (214,204.07) node [anchor=north west][inner sep=0.75pt]  [font=\footnotesize]  {$\varphi _{\downarrow }$};
\draw (213.57,83.92) node [anchor=north west][inner sep=0.75pt]  [font=\footnotesize]  {$\varphi _{\uparrow }$};

\end{tikzpicture}

\end{center}
\noindent where the middle square is a vertical composition of squares of the form $U_i$ and $\psi_i$ as above. The simplest possible form of such decompositions is a decomposition of the form
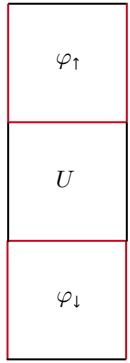
\begin{center}

\tikzset{every picture/.style={line width=0.75pt}} 

\begin{tikzpicture}[x=0.75pt,y=0.75pt,yscale=-1,xscale=1]

\draw [color={rgb, 255:red, 0; green, 0; blue, 0 }  ,draw opacity=1 ]   (211.07,80.3) -- (271.23,80.3) ;
\draw [color={rgb, 255:red, 208; green, 2; blue, 27 }  ,draw opacity=1 ]   (211.07,80.3) -- (211.07,140.47) ;
\draw [color={rgb, 255:red, 208; green, 2; blue, 27 }  ,draw opacity=1 ]   (210.73,140.13) -- (270.9,140.13) ;
\draw [color={rgb, 255:red, 208; green, 2; blue, 27 }  ,draw opacity=1 ]   (270.9,79.97) -- (270.9,140.13) ;
\draw [color={rgb, 255:red, 208; green, 2; blue, 27 }  ,draw opacity=1 ]   (210.73,199.97) -- (270.9,199.97) ;
\draw [color={rgb, 255:red, 208; green, 2; blue, 27 }  ,draw opacity=1 ]   (210.73,199.97) -- (210.73,260.13) ;
\draw [color={rgb, 255:red, 0; green, 0; blue, 0 }  ,draw opacity=1 ]   (210.4,259.8) -- (270.57,259.8) ;
\draw [color={rgb, 255:red, 208; green, 2; blue, 27 }  ,draw opacity=1 ]   (270.57,199.63) -- (270.57,259.8) ;
\draw [color={rgb, 255:red, 0; green, 0; blue, 0 }  ,draw opacity=1 ]   (211.07,140.47) -- (211.07,200.63) ;
\draw [color={rgb, 255:red, 0; green, 0; blue, 0 }  ,draw opacity=1 ]   (270.9,140.13) -- (270.9,200.3) ;

\draw (234,224.07) node [anchor=north west][inner sep=0.75pt]  [font=\footnotesize]  {$\varphi _{\downarrow }$};
\draw (233.57,103.92) node [anchor=north west][inner sep=0.75pt]  [font=\footnotesize]  {$\varphi _{\uparrow }$};
\draw (233.57,162.92) node [anchor=north west][inner sep=0.75pt]  [font=\footnotesize]  {$U$};

\end{tikzpicture}

\end{center}
\noindent Not every square of length 1 in a globularly generated double category admits a simple expression as above. If a square $\varphi$ in a double category $C$ admits an expression as above, we will say that $\varphi$ admits a \textbf{canonical decomposition}, and we will call any decomposition as above a canonical decomposition for $\varphi$. We prove the following lemma.

\begin{lem}\label{lem:canonicalform}
    Let $C$ be a double category. If every non globular square $\varphi$ in $C$, of length 1, admits a canonical decomposition, then $\ell C=1$.
\end{lem}
\begin{proof}
    We prove that the collection of globularly generated non globular squares of length 1 of $C$ is closed under the operation of horizontal composition. Let $\varphi,\varphi'$ be horizontally composable globularly generated non globular squares of length 1 in $C$. Represent $\varphi$ and $\varphi'$ in canonical form
   \begin{center}

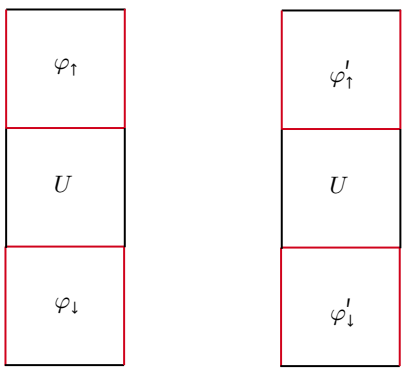
\tikzset{every picture/.style={line width=0.75pt}} 

\begin{tikzpicture}[x=0.75pt,y=0.75pt,yscale=-1,xscale=1]

\draw [color={rgb, 255:red, 0; green, 0; blue, 0 }  ,draw opacity=1 ]   (231.07,100.3) -- (291.23,100.3) ;
\draw [color={rgb, 255:red, 208; green, 2; blue, 27 }  ,draw opacity=1 ]   (231.07,100.3) -- (231.07,160.47) ;
\draw [color={rgb, 255:red, 208; green, 2; blue, 27 }  ,draw opacity=1 ]   (230.73,160.13) -- (290.9,160.13) ;
\draw [color={rgb, 255:red, 208; green, 2; blue, 27 }  ,draw opacity=1 ]   (290.9,99.97) -- (290.9,160.13) ;
\draw [color={rgb, 255:red, 208; green, 2; blue, 27 }  ,draw opacity=1 ]   (230.73,219.97) -- (290.9,219.97) ;
\draw [color={rgb, 255:red, 208; green, 2; blue, 27 }  ,draw opacity=1 ]   (230.73,219.97) -- (230.73,280.13) ;
\draw [color={rgb, 255:red, 0; green, 0; blue, 0 }  ,draw opacity=1 ]   (230.4,279.8) -- (290.57,279.8) ;
\draw [color={rgb, 255:red, 208; green, 2; blue, 27 }  ,draw opacity=1 ]   (290.57,219.63) -- (290.57,279.8) ;
\draw [color={rgb, 255:red, 0; green, 0; blue, 0 }  ,draw opacity=1 ]   (231.07,160.47) -- (231.07,220.63) ;
\draw [color={rgb, 255:red, 0; green, 0; blue, 0 }  ,draw opacity=1 ]   (290.9,160.13) -- (290.9,220.3) ;
\draw [color={rgb, 255:red, 0; green, 0; blue, 0 }  ,draw opacity=1 ]   (370.07,100.8) -- (430.23,100.8) ;
\draw [color={rgb, 255:red, 208; green, 2; blue, 27 }  ,draw opacity=1 ]   (370.07,100.8) -- (370.07,160.97) ;
\draw [color={rgb, 255:red, 208; green, 2; blue, 27 }  ,draw opacity=1 ]   (369.73,160.63) -- (429.9,160.63) ;
\draw [color={rgb, 255:red, 208; green, 2; blue, 27 }  ,draw opacity=1 ]   (429.9,100.47) -- (429.9,160.63) ;
\draw [color={rgb, 255:red, 208; green, 2; blue, 27 }  ,draw opacity=1 ]   (369.73,220.47) -- (429.9,220.47) ;
\draw [color={rgb, 255:red, 208; green, 2; blue, 27 }  ,draw opacity=1 ]   (369.73,220.47) -- (369.73,280.63) ;
\draw [color={rgb, 255:red, 0; green, 0; blue, 0 }  ,draw opacity=1 ]   (369.4,280.3) -- (429.57,280.3) ;
\draw [color={rgb, 255:red, 208; green, 2; blue, 27 }  ,draw opacity=1 ]   (429.57,220.13) -- (429.57,280.3) ;
\draw [color={rgb, 255:red, 0; green, 0; blue, 0 }  ,draw opacity=1 ]   (370.07,160.97) -- (370.07,221.13) ;
\draw [color={rgb, 255:red, 0; green, 0; blue, 0 }  ,draw opacity=1 ]   (429.9,160.63) -- (429.9,220.8) ;

\draw (254,244.07) node [anchor=north west][inner sep=0.75pt]  [font=\footnotesize]  {$\varphi _{\downarrow }$};
\draw (253.57,123.92) node [anchor=north west][inner sep=0.75pt]  [font=\footnotesize]  {$\varphi _{\uparrow }$};
\draw (253.57,182.92) node [anchor=north west][inner sep=0.75pt]  [font=\footnotesize]  {$U$};
\draw (393,244.57) node [anchor=north west][inner sep=0.75pt]  [font=\footnotesize]  {$\varphi '_{\downarrow }$};
\draw (392.57,124.42) node [anchor=north west][inner sep=0.75pt]  [font=\footnotesize]  {$\varphi '_{\uparrow }$};
\draw (392.57,183.42) node [anchor=north west][inner sep=0.75pt]  [font=\footnotesize]  {$U$};

\end{tikzpicture}

   \end{center}
    \noindent The condition that $\varphi,\varphi'$ are horizontally composable means that the black edges of $\varphi$ and $\varphi'$ are equal. The horizontal composition $\varphi\boxvert\varphi'$ is the composition
   \begin{center}

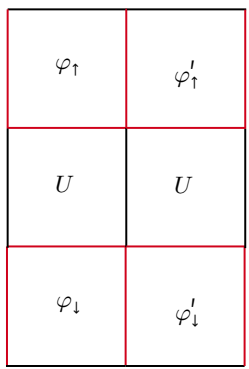
\tikzset{every picture/.style={line width=0.75pt}} 

\begin{tikzpicture}[x=0.75pt,y=0.75pt,yscale=-1,xscale=1]

\draw [color={rgb, 255:red, 0; green, 0; blue, 0 }  ,draw opacity=1 ]   (251.07,109.67) -- (311.23,109.67) ;
\draw [color={rgb, 255:red, 208; green, 2; blue, 27 }  ,draw opacity=1 ]   (251.07,109.67) -- (251.07,169.83) ;
\draw [color={rgb, 255:red, 208; green, 2; blue, 27 }  ,draw opacity=1 ]   (250.73,169.5) -- (310.9,169.5) ;
\draw [color={rgb, 255:red, 208; green, 2; blue, 27 }  ,draw opacity=1 ]   (310.9,109.33) -- (310.9,169.5) ;
\draw [color={rgb, 255:red, 208; green, 2; blue, 27 }  ,draw opacity=1 ]   (250.73,229.33) -- (310.9,229.33) ;
\draw [color={rgb, 255:red, 208; green, 2; blue, 27 }  ,draw opacity=1 ]   (250.73,229.33) -- (250.73,289.5) ;
\draw [color={rgb, 255:red, 0; green, 0; blue, 0 }  ,draw opacity=1 ]   (250.23,289.67) -- (310.4,289.67) ;
\draw [color={rgb, 255:red, 208; green, 2; blue, 27 }  ,draw opacity=1 ]   (310.57,229) -- (310.57,289.17) ;
\draw [color={rgb, 255:red, 0; green, 0; blue, 0 }  ,draw opacity=1 ]   (251.07,169.83) -- (251.07,230) ;
\draw [color={rgb, 255:red, 0; green, 0; blue, 0 }  ,draw opacity=1 ]   (310.9,169.5) -- (310.9,229.67) ;
\draw [color={rgb, 255:red, 0; green, 0; blue, 0 }  ,draw opacity=1 ]   (311.23,109.67) -- (371.4,109.67) ;
\draw [color={rgb, 255:red, 208; green, 2; blue, 27 }  ,draw opacity=1 ]   (310.9,169.5) -- (371.07,169.5) ;
\draw [color={rgb, 255:red, 208; green, 2; blue, 27 }  ,draw opacity=1 ]   (370.9,109.83) -- (370.9,170) ;
\draw [color={rgb, 255:red, 208; green, 2; blue, 27 }  ,draw opacity=1 ]   (310.9,229.33) -- (371.07,229.33) ;
\draw [color={rgb, 255:red, 0; green, 0; blue, 0 }  ,draw opacity=1 ]   (310.4,289.67) -- (370.57,289.67) ;
\draw [color={rgb, 255:red, 208; green, 2; blue, 27 }  ,draw opacity=1 ]   (370.57,229.5) -- (370.57,289.67) ;
\draw [color={rgb, 255:red, 0; green, 0; blue, 0 }  ,draw opacity=1 ]   (370.9,170) -- (370.9,230.17) ;

\draw (274,253.43) node [anchor=north west][inner sep=0.75pt]  [font=\footnotesize]  {$\varphi _{\downarrow }$};
\draw (273.57,133.29) node [anchor=north west][inner sep=0.75pt]  [font=\footnotesize]  {$\varphi _{\uparrow }$};
\draw (273.57,192.29) node [anchor=north west][inner sep=0.75pt]  [font=\footnotesize]  {$U$};
\draw (334,253.93) node [anchor=north west][inner sep=0.75pt]  [font=\footnotesize]  {$\varphi '_{\downarrow }$};
\draw (333.57,133.79) node [anchor=north west][inner sep=0.75pt]  [font=\footnotesize]  {$\varphi '_{\uparrow }$};
\draw (333.57,192.79) node [anchor=north west][inner sep=0.75pt]  [font=\footnotesize]  {$U$};

\end{tikzpicture}

   \end{center}
    \noindent The horiozntal composition of two horizontal identity squares is a horizontal identity square. The above composite square is thus equal to
    \begin{center}

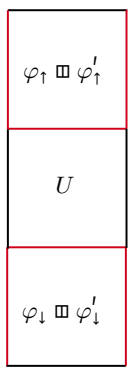
\tikzset{every picture/.style={line width=0.75pt}} 

\begin{tikzpicture}[x=0.75pt,y=0.75pt,yscale=-1,xscale=1]

\draw [color={rgb, 255:red, 0; green, 0; blue, 0 }  ,draw opacity=1 ]   (251.07,110.17) -- (311.23,110.17) ;
\draw [color={rgb, 255:red, 208; green, 2; blue, 27 }  ,draw opacity=1 ]   (251.07,110.17) -- (251.07,170.33) ;
\draw [color={rgb, 255:red, 208; green, 2; blue, 27 }  ,draw opacity=1 ]   (250.73,170) -- (310.9,170) ;
\draw [color={rgb, 255:red, 208; green, 2; blue, 27 }  ,draw opacity=1 ]   (310.9,109.83) -- (310.9,170) ;
\draw [color={rgb, 255:red, 208; green, 2; blue, 27 }  ,draw opacity=1 ]   (250.73,229.83) -- (310.9,229.83) ;
\draw [color={rgb, 255:red, 208; green, 2; blue, 27 }  ,draw opacity=1 ]   (250.73,229.83) -- (250.73,290) ;
\draw [color={rgb, 255:red, 0; green, 0; blue, 0 }  ,draw opacity=1 ]   (250.4,289.67) -- (310.57,289.67) ;
\draw [color={rgb, 255:red, 208; green, 2; blue, 27 }  ,draw opacity=1 ]   (310.57,229.5) -- (310.57,289.67) ;
\draw [color={rgb, 255:red, 0; green, 0; blue, 0 }  ,draw opacity=1 ]   (251.07,170.33) -- (251.07,230.5) ;
\draw [color={rgb, 255:red, 0; green, 0; blue, 0 }  ,draw opacity=1 ]   (310.9,170) -- (310.9,230.17) ;

\draw (256.67,252.93) node [anchor=north west][inner sep=0.75pt]  [font=\footnotesize]  {$\varphi _{\downarrow } \boxvert \varphi '_{\downarrow }$};
\draw (257.24,131.79) node [anchor=north west][inner sep=0.75pt]  [font=\footnotesize]  {$\varphi _{\uparrow } \boxvert \varphi '_{\uparrow }$};
\draw (273.57,192.79) node [anchor=north west][inner sep=0.75pt]  [font=\footnotesize]  {$U$};

\end{tikzpicture}

    \end{center}
    \noindent The horizontal composites $\varphi_\uparrow\boxvert\varphi'_\uparrow$ and $\varphi_\downarrow\boxvert\varphi'_\downarrow$ are globular squares, and thus the above square is of length 1. We conclude that the double category $C$ is of length 1.

\end{proof}

\subsection{Obtaining canonical decomposition}\label{ss:canonicaldecompositions}
\noindent Lemma \ref{lem:canonicalform} provides us with a strategy for proving that a double category $C$ is of length 1: To prove that all the globularly generated squares of $C$ admit a canonical decomposition. The proof of Theorem \ref{thm:main} describes the construction of a double category $\crossb$ associated to $\Phi\in\EndCat$ for a decorated 2-category $\stB$, such that every globular square in $\crossb$ has a canonical decomposition, and thus is of length 1. We describe the idea behind the fact that $\pi_2$-indexings and $\pi_2$-opindexing make our construction work. In order to provide every length 1 non-globular, globularly generated square $\varphi$ of $C$ with a canonical decomposition, we would like to be able to turn every vertical composition of the form (\ref{eq:decomposition}) into a canonical decomposition. The mid-part of a decomposition of the form (\ref{eq:decomposition}) is a vertical composition of squares of the form:
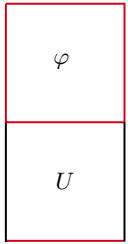
\begin{center}

\tikzset{every picture/.style={line width=0.75pt}} 

\begin{tikzpicture}[x=0.75pt,y=0.75pt,yscale=-1,xscale=1]

\draw [color={rgb, 255:red, 208; green, 2; blue, 27 }  ,draw opacity=1 ]   (211.07,80.3) -- (271.23,80.3) ;
\draw [color={rgb, 255:red, 208; green, 2; blue, 27 }  ,draw opacity=1 ]   (211.07,80.3) -- (211.07,140.47) ;
\draw [color={rgb, 255:red, 208; green, 2; blue, 27 }  ,draw opacity=1 ]   (210.73,140.13) -- (270.9,140.13) ;
\draw [color={rgb, 255:red, 208; green, 2; blue, 27 }  ,draw opacity=1 ]   (270.9,79.97) -- (270.9,140.13) ;
\draw [color={rgb, 255:red, 208; green, 2; blue, 27 }  ,draw opacity=1 ]   (210.73,199.97) -- (270.9,199.97) ;
\draw [color={rgb, 255:red, 0; green, 0; blue, 0 }  ,draw opacity=1 ]   (211.07,140.47) -- (211.07,200.63) ;
\draw [color={rgb, 255:red, 0; green, 0; blue, 0 }  ,draw opacity=1 ]   (270.9,140.13) -- (270.9,200.3) ;

\draw (233.57,103.92) node [anchor=north west][inner sep=0.75pt]  [font=\footnotesize]  {$\varphi $};
\draw (234.57,163.92) node [anchor=north west][inner sep=0.75pt]  [font=\footnotesize]  {$U$};

\end{tikzpicture}

\end{center}
\noindent If we can turn squares as above into squares of the form
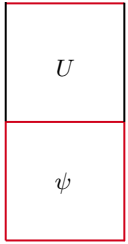
\begin{center}

\tikzset{every picture/.style={line width=0.75pt}} 

\begin{tikzpicture}[x=0.75pt,y=0.75pt,yscale=-1,xscale=1]

\draw [color={rgb, 255:red, 208; green, 2; blue, 27 }  ,draw opacity=1 ]   (231.07,220.3) -- (291.23,220.3) ;
\draw [color={rgb, 255:red, 208; green, 2; blue, 27 }  ,draw opacity=1 ]   (231.07,220.3) -- (231.07,160.13) ;
\draw [color={rgb, 255:red, 208; green, 2; blue, 27 }  ,draw opacity=1 ]   (230.73,160.47) -- (290.9,160.47) ;
\draw [color={rgb, 255:red, 208; green, 2; blue, 27 }  ,draw opacity=1 ]   (290.9,220.63) -- (290.9,160.47) ;
\draw [color={rgb, 255:red, 208; green, 2; blue, 27 }  ,draw opacity=1 ]   (230.73,100.63) -- (290.9,100.63) ;
\draw [color={rgb, 255:red, 0; green, 0; blue, 0 }  ,draw opacity=1 ]   (231.07,160.13) -- (231.07,99.97) ;
\draw [color={rgb, 255:red, 0; green, 0; blue, 0 }  ,draw opacity=1 ]   (290.9,160.47) -- (290.9,100.3) ;

\draw (254.4,184.4) node [anchor=north west][inner sep=0.75pt]  [font=\footnotesize]  {$\psi $};
\draw (254.8,127.8) node [anchor=north west][inner sep=0.75pt]  [font=\footnotesize]  {$U$};

\end{tikzpicture}

\end{center}
\noindent for some other square $\psi$, i.e. if we can slide red boundary squares down along unit squares, then we can, in decomposition (\ref{eq:decomposition}), inductively slide all red boundary squares down all the unit squares of the vertical composition and obtain a composite square of the form
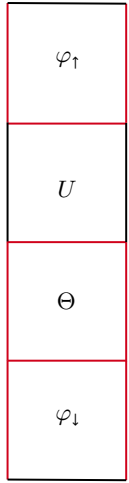
\begin{center}

\tikzset{every picture/.style={line width=0.75pt}} 

\begin{tikzpicture}[x=0.75pt,y=0.75pt,yscale=-1,xscale=1]

\draw [color={rgb, 255:red, 208; green, 2; blue, 27 }  ,draw opacity=1 ]   (251.07,240.3) -- (311.23,240.3) ;
\draw [color={rgb, 255:red, 208; green, 2; blue, 27 }  ,draw opacity=1 ]   (251.07,240.3) -- (251.07,180.13) ;
\draw [color={rgb, 255:red, 208; green, 2; blue, 27 }  ,draw opacity=1 ]   (250.73,180.47) -- (310.9,180.47) ;
\draw [color={rgb, 255:red, 208; green, 2; blue, 27 }  ,draw opacity=1 ]   (310.9,240.63) -- (310.9,180.47) ;
\draw [color={rgb, 255:red, 208; green, 2; blue, 27 }  ,draw opacity=1 ]   (250.73,120.63) -- (310.9,120.63) ;
\draw [color={rgb, 255:red, 0; green, 0; blue, 0 }  ,draw opacity=1 ]   (251.07,180.13) -- (251.07,119.97) ;
\draw [color={rgb, 255:red, 0; green, 0; blue, 0 }  ,draw opacity=1 ]   (310.9,180.47) -- (310.9,120.3) ;
\draw [color={rgb, 255:red, 208; green, 2; blue, 27 }  ,draw opacity=1 ]   (251.07,300.3) -- (251.07,240.13) ;
\draw [color={rgb, 255:red, 208; green, 2; blue, 27 }  ,draw opacity=1 ]   (310.9,300.63) -- (310.9,240.47) ;
\draw    (251.07,300.3) -- (310.9,300.63) ;
\draw [color={rgb, 255:red, 208; green, 2; blue, 27 }  ,draw opacity=1 ]   (251.07,119.94) -- (251.07,59.77) ;
\draw [color={rgb, 255:red, 208; green, 2; blue, 27 }  ,draw opacity=1 ]   (310.9,120.27) -- (310.9,60.11) ;
\draw    (251.07,59.77) -- (310.9,60.11) ;

\draw (274.4,204.4) node [anchor=north west][inner sep=0.75pt]  [font=\footnotesize]  {$\Theta $};
\draw (274.8,147.8) node [anchor=north west][inner sep=0.75pt]  [font=\footnotesize]  {$U$};
\draw (273.67,263.49) node [anchor=north west][inner sep=0.75pt]  [font=\footnotesize]  {$\varphi _{\downarrow }$};
\draw (273.67,83.13) node [anchor=north west][inner sep=0.75pt]  [font=\footnotesize]  {$\varphi _{\uparrow }$};

\end{tikzpicture}

\end{center}
\noindent where $\Theta$ is the vertical composition of all the squares resulting of sliding all red-boundary squares in decomposition (\ref{eq:decomposition}) down unit squares. The vertical composition of the two squares at the bottom of the above decomposition is a square of the form
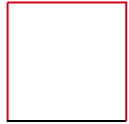
\begin{center}

\tikzset{every picture/.style={line width=0.75pt}} 

\begin{tikzpicture}[x=0.75pt,y=0.75pt,yscale=-1,xscale=1]

\draw [color={rgb, 255:red, 208; green, 2; blue, 27 }  ,draw opacity=1 ]   (370.73,169.97) -- (430.9,169.97) ;
\draw [color={rgb, 255:red, 208; green, 2; blue, 27 }  ,draw opacity=1 ]   (370.73,169.97) -- (370.73,230.13) ;
\draw [color={rgb, 255:red, 0; green, 0; blue, 0 }  ,draw opacity=1 ]   (370.4,229.8) -- (430.57,229.8) ;
\draw [color={rgb, 255:red, 208; green, 2; blue, 27 }  ,draw opacity=1 ]   (430.57,169.63) -- (430.57,229.8) ;

\end{tikzpicture}

\end{center}
\noindent and thus the above decomposition is canonical. $\pi_2$-indexings are precisely the type of structure that allows us to apply the above procedure and thus obtain canonical decomposition for globularly generated squares of length 1. Another way of obtaining canonial decompositions for squares as above is by sliding red-boundary squares up unit squares, as opposed to sliding down. $\pi_2$-opindexings are the structure implementing the operation of sliding squares up.

\subsection{Vertical composition and canonical decompositions}

\noindent Another problem we consider with respect to canonical decompositions is how to express the vertical composition of two non-globular, globularly generated squares, admitting canonical decompositions, as a canonical decomposition. Given two vertically compatible squares
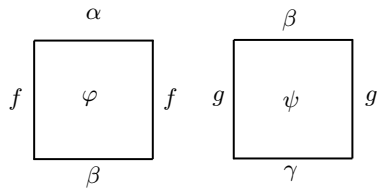
\begin{center}

\tikzset{every picture/.style={line width=0.75pt}} 

\begin{tikzpicture}[x=0.75pt,y=0.75pt,yscale=-1,xscale=1]

\draw [color={rgb, 255:red, 0; green, 0; blue, 0 }  ,draw opacity=1 ]   (251.07,141.3) -- (311.23,141.3) ;
\draw [color={rgb, 255:red, 0; green, 0; blue, 0 }  ,draw opacity=1 ]   (251.07,141.3) -- (251.07,201.47) ;
\draw [color={rgb, 255:red, 0; green, 0; blue, 0 }  ,draw opacity=1 ]   (250.73,201.13) -- (310.9,201.13) ;
\draw [color={rgb, 255:red, 0; green, 0; blue, 0 }  ,draw opacity=1 ]   (310.9,140.97) -- (310.9,201.13) ;
\draw [color={rgb, 255:red, 0; green, 0; blue, 0 }  ,draw opacity=1 ]   (351.73,140.97) -- (411.9,140.97) ;
\draw [color={rgb, 255:red, 0; green, 0; blue, 0 }  ,draw opacity=1 ]   (351.73,140.97) -- (351.73,201.13) ;
\draw [color={rgb, 255:red, 0; green, 0; blue, 0 }  ,draw opacity=1 ]   (351.4,200.8) -- (411.57,200.8) ;
\draw [color={rgb, 255:red, 0; green, 0; blue, 0 }  ,draw opacity=1 ]   (411.57,140.63) -- (411.57,200.8) ;

\draw (375,165.07) node [anchor=north west][inner sep=0.75pt]  [font=\footnotesize]  {$\psi $};
\draw (273.57,164.92) node [anchor=north west][inner sep=0.75pt]  [font=\footnotesize]  {$\varphi $};
\draw (275.57,123.92) node [anchor=north west][inner sep=0.75pt]  [font=\footnotesize]  {$\alpha $};
\draw (275.57,203.4) node [anchor=north west][inner sep=0.75pt]  [font=\footnotesize]  {$\beta $};
\draw (374.57,123.92) node [anchor=north west][inner sep=0.75pt]  [font=\footnotesize]  {$\beta $};
\draw (375.57,202.92) node [anchor=north west][inner sep=0.75pt]  [font=\footnotesize]  {$\gamma $};
\draw (236.57,163.92) node [anchor=north west][inner sep=0.75pt]  [font=\footnotesize]  {$f$};
\draw (314.57,163.92) node [anchor=north west][inner sep=0.75pt]  [font=\footnotesize]  {$f$};
\draw (339.57,164.92) node [anchor=north west][inner sep=0.75pt]  [font=\footnotesize]  {$g$};
\draw (416.57,164.92) node [anchor=north west][inner sep=0.75pt]  [font=\footnotesize]  {$g$};

\end{tikzpicture}

\end{center}
\noindent Assume both $\varphi$ and $\psi$ admit canonical decompositions
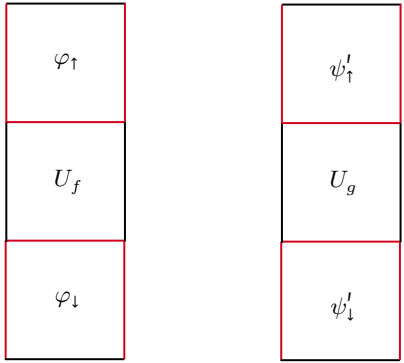
\begin{center}

\tikzset{every picture/.style={line width=0.75pt}} 

\begin{tikzpicture}[x=0.75pt,y=0.75pt,yscale=-1,xscale=1]

\draw [color={rgb, 255:red, 0; green, 0; blue, 0 }  ,draw opacity=1 ]   (229.07,74.67) -- (289.23,74.67) ;
\draw [color={rgb, 255:red, 208; green, 2; blue, 27 }  ,draw opacity=1 ]   (229.07,74.67) -- (229.07,134.83) ;
\draw [color={rgb, 255:red, 208; green, 2; blue, 27 }  ,draw opacity=1 ]   (228.73,134.5) -- (288.9,134.5) ;
\draw [color={rgb, 255:red, 208; green, 2; blue, 27 }  ,draw opacity=1 ]   (288.9,74.33) -- (288.9,134.5) ;
\draw [color={rgb, 255:red, 208; green, 2; blue, 27 }  ,draw opacity=1 ]   (228.73,194.33) -- (288.9,194.33) ;
\draw [color={rgb, 255:red, 208; green, 2; blue, 27 }  ,draw opacity=1 ]   (228.73,194.33) -- (228.73,254.5) ;
\draw [color={rgb, 255:red, 0; green, 0; blue, 0 }  ,draw opacity=1 ]   (228.4,254.17) -- (288.57,254.17) ;
\draw [color={rgb, 255:red, 208; green, 2; blue, 27 }  ,draw opacity=1 ]   (288.57,194) -- (288.57,254.17) ;
\draw [color={rgb, 255:red, 0; green, 0; blue, 0 }  ,draw opacity=1 ]   (229.07,134.83) -- (229.07,195) ;
\draw [color={rgb, 255:red, 0; green, 0; blue, 0 }  ,draw opacity=1 ]   (288.9,134.5) -- (288.9,194.67) ;
\draw [color={rgb, 255:red, 0; green, 0; blue, 0 }  ,draw opacity=1 ]   (368.07,75.17) -- (428.23,75.17) ;
\draw [color={rgb, 255:red, 208; green, 2; blue, 27 }  ,draw opacity=1 ]   (368.07,75.17) -- (368.07,135.33) ;
\draw [color={rgb, 255:red, 208; green, 2; blue, 27 }  ,draw opacity=1 ]   (367.73,135) -- (427.9,135) ;
\draw [color={rgb, 255:red, 208; green, 2; blue, 27 }  ,draw opacity=1 ]   (427.9,74.83) -- (427.9,135) ;
\draw [color={rgb, 255:red, 208; green, 2; blue, 27 }  ,draw opacity=1 ]   (367.73,194.83) -- (427.9,194.83) ;
\draw [color={rgb, 255:red, 208; green, 2; blue, 27 }  ,draw opacity=1 ]   (367.73,194.83) -- (367.73,255) ;
\draw [color={rgb, 255:red, 0; green, 0; blue, 0 }  ,draw opacity=1 ]   (367.4,254.67) -- (427.57,254.67) ;
\draw [color={rgb, 255:red, 208; green, 2; blue, 27 }  ,draw opacity=1 ]   (427.57,194.5) -- (427.57,254.67) ;
\draw [color={rgb, 255:red, 0; green, 0; blue, 0 }  ,draw opacity=1 ]   (368.07,135.33) -- (368.07,195.5) ;
\draw [color={rgb, 255:red, 0; green, 0; blue, 0 }  ,draw opacity=1 ]   (427.9,135) -- (427.9,195.17) ;

\draw (252,218.43) node [anchor=north west][inner sep=0.75pt]  [font=\footnotesize]  {$\varphi _{\downarrow }$};
\draw (251.57,98.29) node [anchor=north west][inner sep=0.75pt]  [font=\footnotesize]  {$\varphi _{\uparrow }$};
\draw (251.57,157.29) node [anchor=north west][inner sep=0.75pt]  [font=\footnotesize]  {$U_{f}$};
\draw (391,218.93) node [anchor=north west][inner sep=0.75pt]  [font=\footnotesize]  {$\psi '_{\downarrow }$};
\draw (390.57,98.79) node [anchor=north west][inner sep=0.75pt]  [font=\footnotesize]  {$\psi '_{\uparrow }$};
\draw (390.57,157.79) node [anchor=north west][inner sep=0.75pt]  [font=\footnotesize]  {$U_{g}$};

\end{tikzpicture}

\end{center}
\noindent The vertical composition $\psi\boxminus\varphi$ is the square
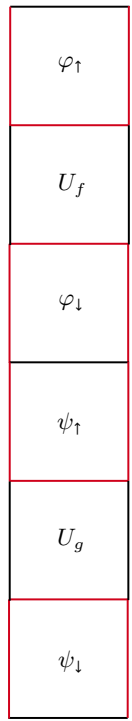
\begin{center}

\tikzset{every picture/.style={line width=0.75pt}} 

\begin{tikzpicture}[x=0.75pt,y=0.75pt,yscale=-1,xscale=1]

\draw [color={rgb, 255:red, 0; green, 0; blue, 0 }  ,draw opacity=1 ]   (209.32,87.67) -- (269.48,87.67) ;
\draw [color={rgb, 255:red, 208; green, 2; blue, 27 }  ,draw opacity=1 ]   (209.32,87.67) -- (209.32,147.83) ;
\draw [color={rgb, 255:red, 208; green, 2; blue, 27 }  ,draw opacity=1 ]   (208.98,147.5) -- (269.15,147.5) ;
\draw [color={rgb, 255:red, 208; green, 2; blue, 27 }  ,draw opacity=1 ]   (269.15,87.33) -- (269.15,147.5) ;
\draw [color={rgb, 255:red, 208; green, 2; blue, 27 }  ,draw opacity=1 ]   (208.98,207.33) -- (269.15,207.33) ;
\draw [color={rgb, 255:red, 208; green, 2; blue, 27 }  ,draw opacity=1 ]   (208.98,207.33) -- (208.98,267.5) ;
\draw [color={rgb, 255:red, 0; green, 0; blue, 0 }  ,draw opacity=1 ]   (208.4,446.67) -- (268.57,446.67) ;
\draw [color={rgb, 255:red, 208; green, 2; blue, 27 }  ,draw opacity=1 ]   (268.82,207) -- (268.82,267.17) ;
\draw [color={rgb, 255:red, 0; green, 0; blue, 0 }  ,draw opacity=1 ]   (209.32,147.83) -- (209.32,208) ;
\draw [color={rgb, 255:red, 0; green, 0; blue, 0 }  ,draw opacity=1 ]   (269.15,147.5) -- (269.15,207.67) ;
\draw [color={rgb, 255:red, 0; green, 0; blue, 0 }  ,draw opacity=1 ]   (209.07,267.17) -- (269.23,267.17) ;
\draw [color={rgb, 255:red, 208; green, 2; blue, 27 }  ,draw opacity=1 ]   (209.07,267.17) -- (209.07,327.33) ;
\draw [color={rgb, 255:red, 208; green, 2; blue, 27 }  ,draw opacity=1 ]   (208.73,327) -- (268.9,327) ;
\draw [color={rgb, 255:red, 208; green, 2; blue, 27 }  ,draw opacity=1 ]   (268.9,266.83) -- (268.9,327) ;
\draw [color={rgb, 255:red, 208; green, 2; blue, 27 }  ,draw opacity=1 ]   (208.73,386.83) -- (268.9,386.83) ;
\draw [color={rgb, 255:red, 208; green, 2; blue, 27 }  ,draw opacity=1 ]   (208.73,386.83) -- (208.73,447) ;
\draw [color={rgb, 255:red, 208; green, 2; blue, 27 }  ,draw opacity=1 ]   (268.57,386.5) -- (268.57,446.67) ;
\draw [color={rgb, 255:red, 0; green, 0; blue, 0 }  ,draw opacity=1 ]   (209.07,327.33) -- (209.07,387.5) ;
\draw [color={rgb, 255:red, 0; green, 0; blue, 0 }  ,draw opacity=1 ]   (268.9,327) -- (268.9,387.17) ;

\draw (232.25,231.43) node [anchor=north west][inner sep=0.75pt]  [font=\footnotesize]  {$\varphi _{\downarrow }$};
\draw (231.82,111.29) node [anchor=north west][inner sep=0.75pt]  [font=\footnotesize]  {$\varphi _{\uparrow }$};
\draw (231.82,170.29) node [anchor=north west][inner sep=0.75pt]  [font=\footnotesize]  {$U_{f}$};
\draw (232,410.93) node [anchor=north west][inner sep=0.75pt]  [font=\footnotesize]  {$\psi _{\downarrow }$};
\draw (231.57,290.79) node [anchor=north west][inner sep=0.75pt]  [font=\footnotesize]  {$\psi _{\uparrow }$};
\draw (231.57,349.79) node [anchor=north west][inner sep=0.75pt]  [font=\footnotesize]  {$U_{g}$};

\end{tikzpicture}

\end{center}
\noindent The red-boundary square $\psi_\uparrow\boxminus\varphi_\downarrow$ is 'trapped' in between unit squares and the square $\varphi_\uparrow$ and $\psi_\downarrow$, and thus the above decomposition is not canonical. Under the action of a $\pi_2$-indexing, as explained in \ref{ss:canonicalform}, we could slide the composite red-boundary squre $\varphi_\downarrow\boxminus\psi_\uparrow$ down and obtain a square of the form
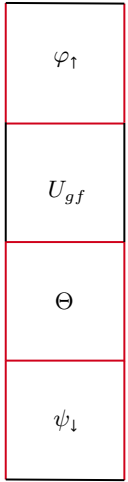
\begin{center}

\tikzset{every picture/.style={line width=0.75pt}} 

\begin{tikzpicture}[x=0.75pt,y=0.75pt,yscale=-1,xscale=1]

\draw [color={rgb, 255:red, 208; green, 2; blue, 27 }  ,draw opacity=1 ]   (271.07,229.67) -- (331.23,229.67) ;
\draw [color={rgb, 255:red, 208; green, 2; blue, 27 }  ,draw opacity=1 ]   (271.07,229.67) -- (271.07,169.5) ;
\draw [color={rgb, 255:red, 208; green, 2; blue, 27 }  ,draw opacity=1 ]   (270.73,169.83) -- (330.9,169.83) ;
\draw [color={rgb, 255:red, 208; green, 2; blue, 27 }  ,draw opacity=1 ]   (330.9,230) -- (330.9,169.83) ;
\draw [color={rgb, 255:red, 208; green, 2; blue, 27 }  ,draw opacity=1 ]   (270.73,110) -- (330.9,110) ;
\draw [color={rgb, 255:red, 0; green, 0; blue, 0 }  ,draw opacity=1 ]   (271.07,169.5) -- (271.07,109.33) ;
\draw [color={rgb, 255:red, 0; green, 0; blue, 0 }  ,draw opacity=1 ]   (330.9,169.83) -- (330.9,109.67) ;
\draw [color={rgb, 255:red, 208; green, 2; blue, 27 }  ,draw opacity=1 ]   (271.07,289.67) -- (271.07,229.5) ;
\draw [color={rgb, 255:red, 208; green, 2; blue, 27 }  ,draw opacity=1 ]   (330.9,290) -- (330.9,229.83) ;
\draw    (271.07,289.67) -- (330.9,290) ;
\draw [color={rgb, 255:red, 208; green, 2; blue, 27 }  ,draw opacity=1 ]   (271.07,109.31) -- (271.07,49.14) ;
\draw [color={rgb, 255:red, 208; green, 2; blue, 27 }  ,draw opacity=1 ]   (330.9,109.64) -- (330.9,49.47) ;
\draw    (271.07,49.14) -- (330.9,49.47) ;

\draw (294.4,193.77) node [anchor=north west][inner sep=0.75pt]  [font=\footnotesize]  {$\Theta $};
\draw (290.8,137.17) node [anchor=north west][inner sep=0.75pt]  [font=\footnotesize]  {$U_{gf}$};
\draw (293.67,252.86) node [anchor=north west][inner sep=0.75pt]  [font=\footnotesize]  {$\psi _{\downarrow }$};
\draw (293.67,72.5) node [anchor=north west][inner sep=0.75pt]  [font=\footnotesize]  {$\varphi _{\uparrow }$};

\end{tikzpicture}

\end{center}
\noindent which again, is canonical. The operation of sliding red-boundary squares up units, i.e. the action of a $\pi_2$-opindexing would yield a similar result. \textbf{Warning}: The ideas presented in this and the previous subsection are the intuitive foundation of our definition of $\pi_2$-indexing and $\pi_2$-opindexing appearing in Section \ref{s:mainthm}, but they do not appear in an essential way in the proof of Theorem \ref{thm:main}. In the proof of Theorem \ref{thm:main}, the operation of sliding squares up and down only makes a short appearance in proving that the double category we construct indeed provides an internalization.

%

\section{Proof of Theorem \ref{thm:main}}\label{sec:internalizationsendindex}
\noindent In this section we present the proof of Theorem \ref{thm:main}. This is done in several steps. We begin with the following observation.

\begin{obs}\label{obs:GrothendieckConstPicture}
    Let $\decB$ be a decorated 2-category. Let $\Phi \in \EndCat$. The delooping embedding $\Omega:\textbf{commMon}\to\Cat$ associating to every commutative monoid $A$ the delooping category $\Omega A$ of $A$, allows us to see $\Phi$ as a functor from $\stB\to\Cat$, where the evaluation on any object $b$ in $\stB$ is $\Omega\pi_2(B,b)$. To simplify notation, we will refer to $\Omega\pi_2(B,b)$ as $\pi_2(B,b)$. We consider the Grothendieck construction $\intGrot$ of $\Phi$. The objects of $\intGrot$ are pairs $(b, \bullet)$ where $b$ is an object in $\stB$, morphisms in $\intGrot$ are pairs of the form $(f, \varphi)$ where $f:b\to c$ is a morphism in $\stB$ and $\varphi$ is a morphism in $\pi_2(B,c)$. We will use the symbol $\boxminus$ for the composition of composable pairs $(g,\phi),(f,\varphi)$ in $\intGrot$, i.e. $(g,\phi)\boxminus(f,\varphi) = (g f, \phi\boxminus \Phi(g)(\varphi))$. We will represent every $(f,\varphi)$ in $\intGrot$ pictorially as:

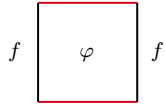
\begin{center}
\tikzset{every picture/.style={line width=0.75pt}} 

\begin{tikzpicture}[x=0.75pt,y=0.75pt,yscale=-1,xscale=1]

\draw [color={rgb, 255:red, 208; green, 2; blue, 27 } ] (100,101.6) -- (150,101.6);
\draw [color={rgb, 255:red, 208; green, 2; blue, 27 } ]  (150,151.6) -- (100,151.6);

\draw (100,101.6) --(100,151.6);
\draw (150,101.6) -- (150,151.6);

\draw (83.36,120.04) node [anchor=north west][inner sep=0.75pt]  [font=\scriptsize]  {$f$};
\draw (155.36,120.04) node [anchor=north west][inner sep=0.75pt]  [font=\scriptsize]  {$f$};
\draw (119.36,123.04) node [anchor=north west][inner sep=0.75pt]  [font=\scriptsize]  {$\phy $};
\end{tikzpicture}     
\end{center}

\noindent Composition in $\intGrot$ is thus represented as:

\begin{center}
\tikzset{every picture/.style={line width=0.75pt}} 

\begin{tikzpicture}[x=0.75pt,y=0.75pt,yscale=-1,xscale=1]

\draw [color={rgb, 255:red, 208; green, 2; blue, 27 }  ,draw opacity=1 ][fill={rgb, 255:red, 208; green, 2; blue, 27 }  ,fill opacity=1 ]   (331,268) -- (399,268.5) ;
\draw [color={rgb, 255:red, 0; green, 0; blue, 0 }  ,draw opacity=1 ]   (331,268) -- (330,337) ;
\draw [color={rgb, 255:red, 208; green, 2; blue, 27 }  ,draw opacity=1 ]   (330,337) -- (398,337.5) ;
\draw [color={rgb, 255:red, 0; green, 0; blue, 0 }  ,draw opacity=1 ]   (399,268.5) -- (398,337.5) ;
\draw [color={rgb, 255:red, 0; green, 0; blue, 0 }  ,draw opacity=1 ]   (330,337) -- (329,406) ;
\draw [color={rgb, 255:red, 208; green, 2; blue, 27 }  ,draw opacity=1 ][fill={rgb, 255:red, 0; green, 0; blue, 0 }  ,fill opacity=1 ]   (329,406) -- (397,406.5) ;
\draw [color={rgb, 255:red, 0; green, 0; blue, 0 }  ,draw opacity=1 ]   (398,337.5) -- (397,406.5) ;
\draw [color={rgb, 255:red, 208; green, 2; blue, 27 }  ,draw opacity=1 ][fill={rgb, 255:red, 208; green, 2; blue, 27 }  ,fill opacity=1 ]   (502,308) -- (611,308.5) ;
\draw [color={rgb, 255:red, 0; green, 0; blue, 0 }  ,draw opacity=1 ]   (502,308) -- (501,377) ;
\draw [color={rgb, 255:red, 208; green, 2; blue, 27 }  ,draw opacity=1 ]   (501,377) -- (610,377.5) ;
\draw [color={rgb, 255:red, 0; green, 0; blue, 0 }  ,draw opacity=1 ]   (611,308.5) -- (610,377.5) ;

\draw (313,292.4) node [anchor=north west][inner sep=0.75pt]  [color={rgb, 255:red, 0; green, 0; blue, 0 }  ,opacity=1 ]  {$f$};
\draw (405,292.4) node [anchor=north west][inner sep=0.75pt]  [color={rgb, 255:red, 0; green, 0; blue, 0 }  ,opacity=1 ]  {$f$};
\draw (356,295.4) node [anchor=north west][inner sep=0.75pt]  [color={rgb, 255:red, 0; green, 0; blue, 0 }  ,opacity=1 ]  {$\varphi $};
\draw (356,366) node [anchor=north west][inner sep=0.75pt]    {$\phi $};
\draw (311,366) node [anchor=north west][inner sep=0.75pt]  [color={rgb, 255:red, 0; green, 0; blue, 0 }  ,opacity=1 ]  {$g$};
\draw (403,366) node [anchor=north west][inner sep=0.75pt]  [color={rgb, 255:red, 0; green, 0; blue, 0 }  ,opacity=1 ]  {$g$};
\draw (426,335.4) node [anchor=north west][inner sep=0.75pt]    {$=$};
\draw (475,333) node [anchor=north west][inner sep=0.75pt]  [color={rgb, 255:red, 0; green, 0; blue, 0 }  ,opacity=1 ]  {$gf$};
\draw (617,333) node [anchor=north west][inner sep=0.75pt]  [color={rgb, 255:red, 0; green, 0; blue, 0 }  ,opacity=1 ]  {$gf$};
\draw (517.3,333.4) node [anchor=north west][inner sep=0.75pt]    {$\phi \boxminus \Phi ( g)( \varphi )$};
\end{tikzpicture}   
\end{center}

\end{obs}\label{obs:Grothendieckconstruction}

\begin{obs}
    Let $\Phi \in \EndCat$: 
    \begin{enumerate}
        \item The projection functor $P:\intGrot \to \stB$ associating $b$ to every object $(b,\bullet)$ in $\intGrot$ and $f$ to every morphism $(f,\varphi)$, is a fiber functor. The functor $U:\stB \to \intGrot$ associating to every object $b$ in $\stB$, $(b,\bullet)$, to every morphism $f$ in $\stB$, $(f, id_c)$ is faithful, and satisfies $PU=id_{\stB}$.
        \item For every object $b$ in $\stB$ there is a faithful functor $\pi_2(B,b)\to \intGrot$ that sends the single object to $(b,\bullet)$ and sends every $\varphi\in \pi_2(B,b)$ to $(id_b,\varphi)$. These functors extend into a functor $\bigsqcup \pi_2(B,b)\to \intGrot$. 
    \end{enumerate}
    
\end{obs}

\begin{definition}\label{def:PhiIntUnionB}
    Let $\decB\in \twocatast$, $\Phi \in \EndCat$. Let $B_1$ denote the category whose objects and morphisms are 1- and 2-cells of $B$ respectively, under vertical composition. The coproduct $\bigsqcup \pi_2(B,b)$ over all objects $b$ in $\stB_0$ is the full subcategory of $B_1$ generated by horizontal identities of objects in $\stB$. We denote by $\GcoppiB$ the pushout of categories:

\begin{center}
\begin{tikzpicture}
\matrix(m)[matrix of math nodes, row sep=4em, column sep=4em,text height=1.5ex, text depth=0.25ex]
{\bigsqcup \pi_2(b,B) & B_1\\
\intGrot & \GcoppiB\\};
\path[->,font=\scriptsize,>=angle 90]
(m-1-1) edge node []{} (m-1-2)
        edge node []{} (m-2-1)
(m-1-2) edge node []{}(m-2-2)
(m-2-1) edge node []{}(m-2-2)
;
\end{tikzpicture}
\end{center}
    
    \noindent The objects of $\GcoppiB$ are the objects of $B_1$ and the objects of $\intGrot$ identifying $(b,\bullet)$ with the horizontal identity of $b$ for all $b\in \stB_0$. Morphisms in $\GcoppiB$ are composites of the morphisms of $B_1$ and the morphisms of $\intGrot$, identifying $\varphi$ with $(id_b, \varphi)$ for all $b\in \stB_0$ and all $\varphi \in \pi_2(B,b)$. Pictorially, morphisms in $\GcoppiB$ are thus composites of squares of either of the following two forms:

\begin{center}

\tikzset{every picture/.style={line width=0.75pt}} 

\begin{tikzpicture}[x=0.75pt,y=0.75pt,yscale=-1,xscale=1]

\draw [color={rgb, 255:red, 208; green, 2; blue, 27 }  ,draw opacity=1 ]   (120,121.6) -- (170,121.6) ;
\draw [color={rgb, 255:red, 208; green, 2; blue, 27 }  ,draw opacity=1 ]   (120,171.6) -- (170,171.6) ;
\draw    (221,120.6) -- (271,120.6) ;
\draw    (221,170.6) -- (271,170.6) ;
\draw    (120,121.6) -- (120,171.6) ;
\draw    (170,121.6) -- (170,171.6) ;
\draw [color={rgb, 255:red, 208; green, 2; blue, 27 }  ,draw opacity=1 ]   (221,120.6) -- (221,170.6) ;
\draw [color={rgb, 255:red, 208; green, 2; blue, 27 }  ,draw opacity=1 ]   (271,120.6) -- (271,170.6) ;

\draw (172.36,140.04) node [anchor=north west][inner sep=0.75pt]  [font=\scriptsize]  {$f$};
\draw (105,140.04) node [anchor=north west][inner sep=0.75pt]  [font=\scriptsize]  {$f$};

\end{tikzpicture}
\end{center}
\noindent where squares on the left are as in Observation \ref{obs:GrothendieckConstPicture}, squares on the right represent 2-cells in $B$, and where the two possible interpretations as above of a square 
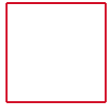
\begin{center}

\tikzset{every picture/.style={line width=0.75pt}} 

\begin{tikzpicture}[x=0.75pt,y=0.75pt,yscale=-1,xscale=1]

\draw [color={rgb, 255:red, 208; green, 2; blue, 27 }  ,draw opacity=1 ]   (241,140.6) -- (291,140.6) ;
\draw [color={rgb, 255:red, 208; green, 2; blue, 27 }  ,draw opacity=1 ]   (241,190.6) -- (291,190.6) ;
\draw [color={rgb, 255:red, 208; green, 2; blue, 27 }  ,draw opacity=1 ]   (241,140.6) -- (241,190.6) ;
\draw [color={rgb, 255:red, 208; green, 2; blue, 27 }  ,draw opacity=1 ]   (291,140.6) -- (291,190.6) ;

\end{tikzpicture}

\end{center}
\noindent are identified. Observe that the composition, in $\GcoppiB$, of 2-cells in $B$ that are not elements of $\pi_2(B,a)$ for some object $a$, and morphisms in $\intGrot$, is a formal concatenation.

\end{definition}
\begin{obs}\label{obs:slidingdowntheorem}
    Let $\decB\in\twocatast$. Let $\Phi\in\EndCat$. Consider two composable morphisms of the form $(id,\varphi)$ and $(f,id)$ in $\intGrot$. We have the identity  $$(f,id)\boxminus(id,\varphi) = (f, \Phi(f)(\varphi))$$\noindent which diagramatically looks like:
\begin{center}

\tikzset{every picture/.style={line width=0.75pt}} 

\begin{tikzpicture}[x=0.75pt,y=0.75pt,yscale=-1,xscale=1]

\draw [color={rgb, 255:red, 208; green, 2; blue, 27 }  ,draw opacity=1 ][fill={rgb, 255:red, 208; green, 2; blue, 27 }  ,fill opacity=1 ]   (349,1501) -- (417,1501) ;
\draw [color={rgb, 255:red, 208; green, 2; blue, 27 }  ,draw opacity=1 ]   (417,1501) -- (417,1569) ;
\draw [color={rgb, 255:red, 208; green, 2; blue, 27 }  ,draw opacity=1 ]   (349,1501) -- (349,1569) ;
\draw [color={rgb, 255:red, 208; green, 2; blue, 27 }  ,draw opacity=1 ][fill={rgb, 255:red, 208; green, 2; blue, 27 }  ,fill opacity=1 ]   (349,1569) -- (417,1569) ;
\draw [color={rgb, 255:red, 0; green, 0; blue, 0 }  ,draw opacity=1 ]   (349,1569) -- (348.59,1596.97) -- (349,1637) ;
\draw [color={rgb, 255:red, 208; green, 2; blue, 27 }  ,draw opacity=1 ]   (349,1637) -- (390,1637) -- (417,1637) ;
\draw [color={rgb, 255:red, 0; green, 0; blue, 0 }  ,draw opacity=1 ]   (417,1569) -- (417,1637) ;
\draw [color={rgb, 255:red, 208; green, 2; blue, 27 }  ,draw opacity=1 ][fill={rgb, 255:red, 208; green, 2; blue, 27 }  ,fill opacity=1 ]   (513,1536) -- (581,1536) ;
\draw [color={rgb, 255:red, 0; green, 0; blue, 0 }  ,draw opacity=1 ]   (581,1536) -- (581,1604) ;
\draw [color={rgb, 255:red, 208; green, 2; blue, 27 }  ,draw opacity=1 ]   (513,1604) -- (581,1604) ;
\draw [color={rgb, 255:red, 0; green, 0; blue, 0 }  ,draw opacity=1 ]   (513,1536) -- (513,1604) ;

\draw (329,1592.4) node [anchor=north west][inner sep=0.75pt]  [color={rgb, 255:red, 0; green, 0; blue, 0 }  ,opacity=1 ]  {$f$};
\draw (428,1592.4) node [anchor=north west][inner sep=0.75pt]  [color={rgb, 255:red, 0; green, 0; blue, 0 }  ,opacity=1 ]  {$f$};
\draw (376,1595.4) node [anchor=north west][inner sep=0.75pt]  [color={rgb, 255:red, 0; green, 0; blue, 0 }  ,opacity=1 ]  {$id$};
\draw (376,1526.4) node [anchor=north west][inner sep=0.75pt]    {$\varphi $};
\draw (454,1560.4) node [anchor=north west][inner sep=0.75pt]    {$=$};
\draw (491,1560.4) node [anchor=north west][inner sep=0.75pt]  [color={rgb, 255:red, 0; green, 0; blue, 0 }  ,opacity=1 ]  {$f$};
\draw (589,1560.4) node [anchor=north west][inner sep=0.75pt]  [color={rgb, 255:red, 0; green, 0; blue, 0 }  ,opacity=1 ]  {$f$};
\draw (519,1561.4) node [anchor=north west][inner sep=0.75pt]  [color={rgb, 255:red, 0; green, 0; blue, 0 }  ,opacity=1 ]  {$\Phi ( f)( \varphi )$};

\end{tikzpicture}
\end{center}
    \noindent Now, observe that for every $(f,\phi)$ in $\intGrot$ we have 
$$(f,\phi)=(f,\phi\boxminus id)=(id\circ f, \phi\boxminus\Phi (id)(id))=(id,\phi)\boxminus(f,id).$$ Diagramatically, the above equation is the composition:

\begin{center}

\tikzset{every picture/.style={line width=0.75pt}} 

\begin{tikzpicture}[x=0.75pt,y=0.75pt,yscale=-1,xscale=1]

\draw    (120,100.5) -- (120,170.6) ;
\draw    (190,100.5) -- (190,170.6) ;
\draw [color={rgb, 255:red, 208; green, 2; blue, 27 }  ,draw opacity=1 ]   (120,100.5) -- (190,100.5) ;
\draw [color={rgb, 255:red, 208; green, 2; blue, 27 }  ,draw opacity=1 ]   (120,170.6) -- (190,170.6) ;
\draw    (280.71,64.9) -- (280.71,135) ;
\draw    (350.71,64.9) -- (350.71,135) ;
\draw [color={rgb, 255:red, 208; green, 2; blue, 27 }  ,draw opacity=1 ]   (280.71,64.9) -- (350.71,64.9) ;
\draw [color={rgb, 255:red, 208; green, 2; blue, 27 }  ,draw opacity=1 ]   (280.71,135) -- (350.71,135) ;
\draw [color={rgb, 255:red, 208; green, 2; blue, 27 }  ,draw opacity=1 ]   (280.71,135.07) -- (280.71,205.17) ;
\draw [color={rgb, 255:red, 208; green, 2; blue, 27 }  ,draw opacity=1 ]   (350.71,135.07) -- (350.71,205.17) ;
\draw [color={rgb, 255:red, 208; green, 2; blue, 27 }  ,draw opacity=1 ]   (280.71,205.17) -- (350.71,205.17) ;

\draw (99,125.26) node [anchor=north west][inner sep=0.75pt]    {$f$};
\draw (199,124.11) node [anchor=north west][inner sep=0.75pt]    {$f$};
\draw (147.86,124.69) node [anchor=north west][inner sep=0.75pt]    {$\phi $};
\draw (308.57,89.09) node [anchor=north west][inner sep=0.75pt]    {$id$};
\draw (308.57,159.26) node [anchor=north west][inner sep=0.75pt]    {$\phi $};
\draw (260.14,92.4) node [anchor=north west][inner sep=0.75pt]    {$f$};
\draw (359.29,92.4) node [anchor=north west][inner sep=0.75pt]    {$f$};
\draw (228.43,125.26) node [anchor=north west][inner sep=0.75pt]    {$=$};

\end{tikzpicture}

\end{center}

\noindent Combining the last two diagrammatic equations, we have the relation $(f,id)\boxminus (id,\varphi) = (id,\Phi(f)(\varphi))\boxminus (f,id)$. Diagramatically:
\begin{center}

\tikzset{every picture/.style={line width=0.75pt}} 

\begin{tikzpicture}[x=0.75pt,y=0.75pt,yscale=-1,xscale=1]

\draw [color={rgb, 255:red, 208; green, 2; blue, 27 }  ,draw opacity=1 ][fill={rgb, 255:red, 208; green, 2; blue, 27 }  ,fill opacity=1 ]   (52,1700) -- (120,1700) ;
\draw [color={rgb, 255:red, 208; green, 2; blue, 27 }  ,draw opacity=1 ]   (120,1700) -- (120,1768) ;
\draw [color={rgb, 255:red, 208; green, 2; blue, 27 }  ,draw opacity=1 ]   (52,1700) -- (52,1768) ;
\draw [color={rgb, 255:red, 208; green, 2; blue, 27 }  ,draw opacity=1 ][fill={rgb, 255:red, 208; green, 2; blue, 27 }  ,fill opacity=1 ]   (52,1768) -- (120,1768) ;
\draw [color={rgb, 255:red, 0; green, 0; blue, 0 }  ,draw opacity=1 ]   (52,1768) -- (51.59,1795.97) -- (52,1836) ;
\draw [color={rgb, 255:red, 208; green, 2; blue, 27 }  ,draw opacity=1 ]   (52,1836) -- (93,1836) -- (120,1836) ;
\draw [color={rgb, 255:red, 0; green, 0; blue, 0 }  ,draw opacity=1 ]   (120,1768) -- (120,1836) ;
\draw [color={rgb, 255:red, 208; green, 2; blue, 27 }  ,draw opacity=1 ][fill={rgb, 255:red, 208; green, 2; blue, 27 }  ,fill opacity=1 ]   (209,1702) -- (277,1702) ;
\draw [color={rgb, 255:red, 0; green, 0; blue, 0 }  ,draw opacity=1 ]   (277,1702) -- (277,1770) ;
\draw [color={rgb, 255:red, 0; green, 0; blue, 0 }  ,draw opacity=1 ]   (209,1702) -- (209,1770) ;
\draw [color={rgb, 255:red, 208; green, 2; blue, 27 }  ,draw opacity=1 ]   (209,1770) -- (208.59,1797.97) -- (209,1838) ;
\draw [color={rgb, 255:red, 208; green, 2; blue, 27 }  ,draw opacity=1 ]   (209,1838) -- (277,1838) ;
\draw [color={rgb, 255:red, 208; green, 2; blue, 27 }  ,draw opacity=1 ]   (277,1770) -- (277,1838) ;
\draw [color={rgb, 255:red, 208; green, 2; blue, 27 }  ,draw opacity=1 ][fill={rgb, 255:red, 208; green, 2; blue, 27 }  ,fill opacity=1 ]   (209,1770) -- (277,1770) ;

\draw (32,1791.4) node [anchor=north west][inner sep=0.75pt]  [color={rgb, 255:red, 0; green, 0; blue, 0 }  ,opacity=1 ]  {$f$};
\draw (131,1791.4) node [anchor=north west][inner sep=0.75pt]  [color={rgb, 255:red, 0; green, 0; blue, 0 }  ,opacity=1 ]  {$f$};
\draw (79,1794.4) node [anchor=north west][inner sep=0.75pt]  [color={rgb, 255:red, 0; green, 0; blue, 0 }  ,opacity=1 ]  {$id$};
\draw (79,1725.4) node [anchor=north west][inner sep=0.75pt]    {$\varphi $};
\draw (157,1759.4) node [anchor=north west][inner sep=0.75pt]    {$=$};
\draw (190,1724.4) node [anchor=north west][inner sep=0.75pt]  [color={rgb, 255:red, 0; green, 0; blue, 0 }  ,opacity=1 ]  {$f$};
\draw (282,1724.4) node [anchor=north west][inner sep=0.75pt]  [color={rgb, 255:red, 0; green, 0; blue, 0 }  ,opacity=1 ]  {$f$};
\draw (234,1727.4) node [anchor=north west][inner sep=0.75pt]    {$id$};
\draw (215,1795.4) node [anchor=north west][inner sep=0.75pt]  [color={rgb, 255:red, 0; green, 0; blue, 0 }  ,opacity=1 ]  {$\Phi ( f)( \varphi )$};

\end{tikzpicture}
\end{center}
\noindent The above equation says that the $\pi_2$-indexing $\Phi$ gives us a way to slide squares down along squares parametrized by an $f$. See the comments in Subsection \ref{ss:canonicaldecompositions}. 
\end{obs}
\noindent In the following observation we analyze the composition of morphisms in $\intGrot\sqcup_\pi B_1$.
\begin{obs}\label{obs:compositionpi}
     In order to spell out the composition operation in $\intGrot\sqcup_\pi B_1$ we consider several cases. \textbf{Case 1}: The two morphisms are in $B_1$. In that case the composition 
    is the composition in $B_1$. \textbf{Case 2}: The two morphisms are in $\int_{\stB}\Phi$. In that case the composition is the composition in Observation \ref{obs:GrothendieckConstPicture}. \textbf{Case 3}: The first morphism is a morphism $\varphi$ in $B_1$ but not in $\int_{\stB}\Phi$ and the second morphism is a morphism $(f,\psi)$ in $\int_{\stB}\Phi$ but not in $B_1$. Then the composition $(f,\psi)\boxminus \varphi$ is the formal concatenation of $(f,\psi)$ and $\varphi$. \textbf{Case 4}: The first morphism is a morphism $(f,\varphi)$ in $\int_{\stB}\Phi$ but not in $B_1$ and the second morphism is a morphism $\psi$ in $B_1$ but not in $\int_{\stB}\Phi$. Then, we have the next equation for the composition of $\psi$ with $(f,\varphi)$: $$\psi \boxminus (f,\varphi) = \psi\boxminus [(id, \varphi)\boxminus(f,id)]=[\psi\boxminus\varphi]\boxminus (f,id)$$
Thus, geometrically, the composition of $\psi$ and $(f,\varphi)$ is the concatenation:

\begin{center}

\tikzset{every picture/.style={line width=0.75pt}} 

\begin{tikzpicture}[x=0.75pt,y=0.75pt,yscale=-1,xscale=1]

\draw    (300.71,84.9) -- (300.71,155) ;
\draw    (370.71,84.9) -- (370.71,155) ;
\draw [color={rgb, 255:red, 208; green, 2; blue, 27 }  ,draw opacity=1 ]   (300.71,84.9) -- (370.71,84.9) ;
\draw [color={rgb, 255:red, 208; green, 2; blue, 27 }  ,draw opacity=1 ]   (300.71,155) -- (370.71,155) ;
\draw [color={rgb, 255:red, 208; green, 2; blue, 27 }  ,draw opacity=1 ]   (300.71,155.07) -- (300.71,225.17) ;
\draw [color={rgb, 255:red, 208; green, 2; blue, 27 }  ,draw opacity=1 ]   (370.71,155.07) -- (370.71,225.17) ;
\draw [color={rgb, 255:red, 0; green, 0; blue, 0 }  ,draw opacity=1 ]   (300.71,225.17) -- (370.71,225.17) ;
\draw    (149.91,84.73) -- (149.91,154.83) ;
\draw    (219.91,84.73) -- (219.91,154.83) ;
\draw [color={rgb, 255:red, 208; green, 2; blue, 27 }  ,draw opacity=1 ]   (149.91,84.73) -- (219.91,84.73) ;
\draw [color={rgb, 255:red, 208; green, 2; blue, 27 }  ,draw opacity=1 ]   (149.91,154.83) -- (219.91,154.83) ;
\draw [color={rgb, 255:red, 208; green, 2; blue, 27 }  ,draw opacity=1 ]   (149.91,154.9) -- (149.91,225) ;
\draw [color={rgb, 255:red, 208; green, 2; blue, 27 }  ,draw opacity=1 ]   (219.91,154.9) -- (219.91,225) ;
\draw [color={rgb, 255:red, 0; green, 0; blue, 0 }  ,draw opacity=1 ]   (149.91,225) -- (219.91,225) ;

\draw (328.57,109.09) node [anchor=north west][inner sep=0.75pt]    {$id$};
\draw (312.57,179.26) node [anchor=north west][inner sep=0.75pt]    {$\varphi \boxminus \psi $};
\draw (280.14,112.4) node [anchor=north west][inner sep=0.75pt]    {$f$};
\draw (379.29,112.4) node [anchor=north west][inner sep=0.75pt]    {$f$};
\draw (248.43,145.26) node [anchor=north west][inner sep=0.75pt]    {$=$};
\draw (177.77,108.91) node [anchor=north west][inner sep=0.75pt]    {$\varphi $};
\draw (177.77,179.09) node [anchor=north west][inner sep=0.75pt]    {$\psi $};
\draw (129.34,112.23) node [anchor=north west][inner sep=0.75pt]    {$f$};
\draw (228.49,112.23) node [anchor=north west][inner sep=0.75pt]    {$f$};

\end{tikzpicture}

\end{center}
\end{obs}

\noindent We now proceed to the proof of Theorem \ref{thm:main}. We begin with the construction of the double category $\crossb$.

\

\noindent \textbf{Construction of $\crossb$}: Let $\decB\in\twocatast$ and $\Phi\in \EndCat$. The construction of $\crossb$ will be performed in a series of steps: 

\

\noindent \textbf{1. Underlying categories:} We make the category of objects $\crossb_0$ of $\crossb$ be $\stB$. We make the category of morphisms $\crossb_1$ of $\crossb$ be the category $\GcoppiB$ of Definition \ref{def:PhiIntUnionB}. 

\

\noindent \textbf{2. Unit functor:} The unit functor $U:\crossb_0\to\crossb_1$ of $\crossb$ will be the functor:

\begin{center}
\begin{tikzpicture}
\matrix(m)[matrix of math nodes, row sep=4em, column sep=2.5em,text height=1.5ex, text depth=0.25ex]
{\stB & \intGrot & \GcopB & \GcoppiB\\};
\path[->,font=\scriptsize,>=angle 90]
(m-1-1) edge node [above]{$U$} (m-1-2)
(m-1-3) edge node [above]{$\rho$} (m-1-4)
;
\path[->,font=\scriptsize,>=angle 90]
(m-1-2) edge node [above]{} (m-1-3)
;
\end{tikzpicture}
\end{center}
    
\noindent Given a morphism $f\in \decB$, the square $U(f)$ is thus defined by the pictorial equation:

\begin{center}
\tikzset{every picture/.style={line width=0.75pt}} 

\begin{tikzpicture}[x=0.75pt,y=0.75pt,yscale=-1,xscale=1]

\draw [color={rgb, 255:red, 208; green, 2; blue, 27 }  ,draw opacity=1 ][fill={rgb, 255:red, 208; green, 2; blue, 27 }  ,fill opacity=1 ]   (28,105) -- (96,105.5) ;
\draw [color={rgb, 255:red, 0; green, 0; blue, 0 }  ,draw opacity=1 ]   (28,105) -- (27,174) ;
\draw [color={rgb, 255:red, 208; green, 2; blue, 27 }  ,draw opacity=1 ]   (27,174) -- (95,174.5) ;
\draw [color={rgb, 255:red, 0; green, 0; blue, 0 }  ,draw opacity=1 ]   (96,105.5) -- (95,174.5) ;
\draw [color={rgb, 255:red, 208; green, 2; blue, 27 }  ,draw opacity=1 ][fill={rgb, 255:red, 208; green, 2; blue, 27 }  ,fill opacity=1 ]   (177,108) -- (245,108.5) ;
\draw [color={rgb, 255:red, 0; green, 0; blue, 0 }  ,draw opacity=1 ]   (177,108) -- (176,177) ;
\draw [color={rgb, 255:red, 208; green, 2; blue, 27 }  ,draw opacity=1 ]   (176,177) -- (244,177.5) ;
\draw [color={rgb, 255:red, 0; green, 0; blue, 0 }  ,draw opacity=1 ]   (245,108.5) -- (244,177.5) ;

\draw (10,129.4) node [anchor=north west][inner sep=0.75pt]  [color={rgb, 255:red, 0; green, 0; blue, 0 }  ,opacity=1 ]  {$f$};
\draw (102,130.4) node [anchor=north west][inner sep=0.75pt]  [color={rgb, 255:red, 0; green, 0; blue, 0 }  ,opacity=1 ]  {$f$};
\draw (54,132.4) node [anchor=north west][inner sep=0.75pt]    {$U$};
\draw (159,132.4) node [anchor=north west][inner sep=0.75pt]  [color={rgb, 255:red, 0; green, 0; blue, 0 }  ,opacity=1 ]  {$f$};
\draw (251,133.4) node [anchor=north west][inner sep=0.75pt]  [color={rgb, 255:red, 0; green, 0; blue, 0 }  ,opacity=1 ]  {$f$};
\draw (201.4,135.4) node [anchor=north west][inner sep=0.75pt]    {$id$};
\draw (122,132.4) node [anchor=north west][inner sep=0.75pt]    {$:=$};
\end{tikzpicture}
\end{center}

\

\noindent \textbf{3. Frame functors:} Observe that there is a functor $\intGrot\sqcup B_1\to \stB$ given by the universal property of the coproduct applied to the fibred functor $P:\intGrot \to \stB$ and the domain function $dom: B_1\to \stB$. Clearly we have that $P((id,\varphi))=dom(\varphi)=id$, then we obtain a left frame functor $L: \intGrot\sqcup_\pi  B \to \stB$. The right frame functor $R:\GcoppiB \to \stB$ for $\crossb$ is obtained analogously. These two functors match with the convention for the pictorial representation of $(f,\varphi)$ in Observation \ref{obs:GrothendieckConstPicture}. 

\

\noindent \textbf{4. Horizontal composition:} We define the horizontal composition functor for $\crossb$ in parts. First consider the functor $\boxvert:\intGrot\times_{\stB}\intGrot \to \intGrot$ defined as $(f, \varphi) \boxvert (f, \psi)=(f, \psi\boxvert\varphi)$. Pictorially:

\begin{center}
\tikzset{every picture/.style={line width=0.75pt}} 

\begin{tikzpicture}[x=0.75pt,y=0.75pt,yscale=-1,xscale=1]

\draw [color={rgb, 255:red, 208; green, 2; blue, 27 }  ,draw opacity=1 ][fill={rgb, 255:red, 208; green, 2; blue, 27 }  ,fill opacity=1 ]   (36,431) -- (104,431) ;
\draw [color={rgb, 255:red, 0; green, 0; blue, 0 }  ,draw opacity=1 ]   (36,431) -- (36,499) ;
\draw [color={rgb, 255:red, 208; green, 2; blue, 27 }  ,draw opacity=1 ]   (36,499) -- (104,499) ;
\draw [color={rgb, 255:red, 0; green, 0; blue, 0 }  ,draw opacity=1 ]   (104,431) -- (104,499) ;
\draw [color={rgb, 255:red, 208; green, 2; blue, 27 }  ,draw opacity=1 ][fill={rgb, 255:red, 208; green, 2; blue, 27 }  ,fill opacity=1 ]   (104,431) -- (172,431) ;
\draw [color={rgb, 255:red, 208; green, 2; blue, 27 }  ,draw opacity=1 ]   (104,499) -- (172,499) ;
\draw [color={rgb, 255:red, 0; green, 0; blue, 0 }  ,draw opacity=1 ]   (172,431) -- (172,499) ;
\draw [color={rgb, 255:red, 208; green, 2; blue, 27 }  ,draw opacity=1 ][fill={rgb, 255:red, 208; green, 2; blue, 27 }  ,fill opacity=1 ]   (252,434) -- (320,434) ;
\draw [color={rgb, 255:red, 0; green, 0; blue, 0 }  ,draw opacity=1 ]   (252,434) -- (252,502) ;
\draw [color={rgb, 255:red, 208; green, 2; blue, 27 }  ,draw opacity=1 ]   (252,502) -- (320,502) ;
\draw [color={rgb, 255:red, 0; green, 0; blue, 0 }  ,draw opacity=1 ]   (320,434) -- (320,502) ;

\draw (15,453.4) node [anchor=north west][inner sep=0.75pt]  [color={rgb, 255:red, 0; green, 0; blue, 0 }  ,opacity=1 ]  {$f$};
\draw (176,455.4) node [anchor=north west][inner sep=0.75pt]  [color={rgb, 255:red, 0; green, 0; blue, 0 }  ,opacity=1 ]  {$f$};
\draw (59,455.4) node [anchor=north west][inner sep=0.75pt]    {$\varphi $};
\draw (85,455.4) node [anchor=north west][inner sep=0.75pt]  [color={rgb, 255:red, 0; green, 0; blue, 0 }  ,opacity=1 ]  {$f$};
\draw (130,454.4) node [anchor=north west][inner sep=0.75pt]    {$\psi $};
\draw (200,458.4) node [anchor=north west][inner sep=0.75pt]    {$=$};
\draw (235,455.4) node [anchor=north west][inner sep=0.75pt]  [color={rgb, 255:red, 0; green, 0; blue, 0 }  ,opacity=1 ]  {$f$};
\draw (324,455.4) node [anchor=north west][inner sep=0.75pt]  [color={rgb, 255:red, 0; green, 0; blue, 0 }  ,opacity=1 ]  {$f$};
\draw (266.4,456.4) node [anchor=north west][inner sep=0.75pt]    {$\psi \boxvert \varphi $};
\end{tikzpicture}
\end{center}

\noindent Horizontal composition of $B$ defines a functor $\boxvert: B_1 \times_{B^*} B_1 \to B_1$. Observe that $(id,\varphi)\boxvert(id,\psi) = (id, \psi\boxvert\varphi)=\varphi\boxvert \psi$. The two composition operations defined above thus coincide in $\bigsqcup\pi_2(B,b)$. We obtain a horizontal composition functor $$\boxvert: \left( \GcoppiB \right) \times_{B^*} \left(\GcoppiB \right)  \to \GcoppiB$$
    \noindent making the following diagram commute:

\begin{center}
\begin{tikzpicture}
\matrix(m)[matrix of math nodes, row sep=4em, column sep=2em,text height=1.5ex, text depth=0.25ex]
{\intGrot\times_{\stB}\intGrot & \left(\GcoppiB\right)\times_{\stB}\left(\GcoppiB\right) & B_1\times_{\stB} B_1\\
\intGrot & \left(\GcoppiB\right) & B_1\\};
\path[->,font=\scriptsize,>=angle 90]
(m-1-1) edge node []{} (m-1-2)
        edge node [left]{$\boxvert$} (m-2-1)
(m-2-1) edge node []{}(m-2-2)
(m-1-2) edge node [left]{$\boxvert$}(m-2-2)
(m-1-3) edge node [right]{$\boxvert$}(m-2-3)
        edge node []{} (m-1-2)
(m-2-3) edge node []{}(m-2-2)
;
\end{tikzpicture}
\end{center}

    \noindent \textbf{5. Relations:} We check that the relations between the structure functors defined above, making $\crossb$ into a double category are satisfied. Let $f$ be a morphism in $B^*$, by definition of $L$ and $U$ we have that $LU(f) = L((f,id))= f$, so $LU = id_{B^*}$. In the same way, we have the equation $RU=id_{B^*}$. Now, let $\varphi, \psi$ two morphisms in $\GcoppiB$ such that $R(\varphi)$=$L(\psi)$. If $\varphi,\psi$ are both in $B_1$, then $L(\varphi\boxvert \psi) = L(\varphi)$ and $R(\varphi\boxvert \psi) = R(\psi)$. Similarly, if $\varphi,\psi$ are in $\int_{B^*} \Phi$ by definition of $\boxvert$ we have that that $L(\varphi\boxvert \psi) = L(\varphi)$ and $R(\varphi\boxvert \psi) = R(\psi)$. We conclude that $L\boxvert = L\pi_l$ and $R\boxvert=R\pi_r$ where $\pi_l$ is and $\pi_r$ are the projections of the fibered product. We conclude that with the structure defined above $\crossb$ is a double category.

\

\noindent \textbf{$\crossb$ is a globularly generated internalization of length 1:} We finalize the proof of Theorem \ref{thm:main} by proving that the double category $\crossb$ indeed is a globularly generated internalization of $\decB$ of length 1. We first prove that $\crossb$ is globularly generated and that $\ell(\crossb )=1$. By Observaton \ref{obs:slidingdowntheorem}, every square appearing in Definition \ref{def:PhiIntUnionB} admits a decomposition in canonical form. By Lemma \ref{lem:canonicalform}, the double category $\crossb$ is globularly generated and of length 1. We now prove the equation $H^\ast (\crossb)=\decB$. Squares in $\crossb$ are either squares of the form
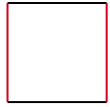
\begin{center}

\tikzset{every picture/.style={line width=0.75pt}} 

\begin{tikzpicture}[x=0.75pt,y=0.75pt,yscale=-1,xscale=1]

\draw    (241,140.6) -- (291,140.6) ;
\draw    (241,190.6) -- (291,190.6) ;
\draw [color={rgb, 255:red, 208; green, 2; blue, 27 }  ,draw opacity=1 ]   (241,140.6) -- (241,190.6) ;
\draw [color={rgb, 255:red, 208; green, 2; blue, 27 }  ,draw opacity=1 ]   (291,140.6) -- (291,190.6) ;

\end{tikzpicture}

\end{center}
\noindent representing 2-cells in $B$, or admit decompositions of the form
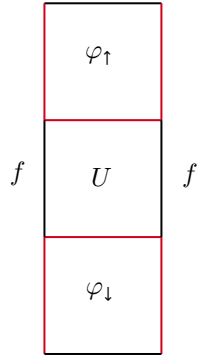
\begin{center}
\tikzset{every picture/.style={line width=0.75pt}} 

\begin{tikzpicture}[x=0.75pt,y=0.75pt,yscale=-1,xscale=1]

\draw    (521,851) -- (580,851) ;
\draw [color={rgb, 255:red, 208; green, 2; blue, 27 }  ,draw opacity=1 ][fill={rgb, 255:red, 208; green, 2; blue, 27 }  ,fill opacity=1 ]   (521,851) -- (521,910) ;
\draw [color={rgb, 255:red, 208; green, 2; blue, 27 }  ,draw opacity=1 ][fill={rgb, 255:red, 208; green, 2; blue, 27 }  ,fill opacity=1 ]   (580,851) -- (580,910) ;
\draw [color={rgb, 255:red, 208; green, 2; blue, 27 }  ,draw opacity=1 ][fill={rgb, 255:red, 208; green, 2; blue, 27 }  ,fill opacity=1 ]   (521,910) -- (580,910) ;
\draw [color={rgb, 255:red, 208; green, 2; blue, 27 }  ,draw opacity=1 ][fill={rgb, 255:red, 208; green, 2; blue, 27 }  ,fill opacity=1 ]   (521,969) -- (580,969) ;
\draw [color={rgb, 255:red, 208; green, 2; blue, 27 }  ,draw opacity=1 ][fill={rgb, 255:red, 208; green, 2; blue, 27 }  ,fill opacity=1 ]   (580,1028) -- (580,969) ;
\draw [color={rgb, 255:red, 208; green, 2; blue, 27 }  ,draw opacity=1 ][fill={rgb, 255:red, 208; green, 2; blue, 27 }  ,fill opacity=1 ]   (521,1028) -- (521,969) ;
\draw    (521,1028) -- (580,1028) ;
\draw    (580,910) -- (580,969) ;
\draw    (521,910) -- (521,969) ;

\draw (543,932.4) node [anchor=north west][inner sep=0.75pt]  [color={rgb, 255:red, 0; green, 0; blue, 0 }  ,opacity=1 ]  {$U$};
\draw (589,930.4) node [anchor=north west][inner sep=0.75pt]  [color={rgb, 255:red, 0; green, 0; blue, 0 }  ,opacity=1 ]  {$f$};
\draw (502,929.4) node [anchor=north west][inner sep=0.75pt]  [color={rgb, 255:red, 0; green, 0; blue, 0 }  ,opacity=1 ]  {$f$};
\draw (540,871.4) node [anchor=north west][inner sep=0.75pt]    {$\varphi_\uparrow$};
\draw (540.4,990.4) node [anchor=north west][inner sep=0.75pt]    {$\varphi _\downarrow$};
\end{tikzpicture}
\end{center}
\noindent A square as above is globular if and only if $f$ is an identity, in which case it would be a vertical composition of 2-cells in $B$ and thus would itself be a 2-cell in $B$. This proves the equation $H^\ast \crossb=\decB$. This concludes the proof of the Theorem. 
$\square$

\begin{note} A construction analogous to that appearing in Observation \ref{obs:compositionpi} can be associated to an opindexing. The construction and the arguments appearing in the proof of Theorem \ref{thm:main} for the case of opindexings is completely analogous as to those presented above.
\end{note}

\section{Examples}\label{s:examples}
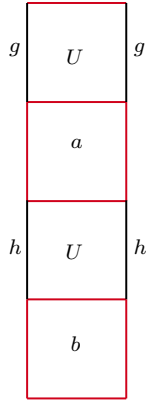
\begin{ex}\label{ex:groupsdoublecat}
    Let $G$ be a group. Let $A$ be an abelian group. Let $\Phi\in\pi_2\textbf{Ind}_{(\Omega G,2\Omega A)}$. Example \ref{ex:abeliangroups} says that $\Phi$ can be identified with an action $\Phi:G\acts A$ of $G$ on $A$ by automorphisms. The 2-category $2\Omega A$ has a single 0- and 1-cell. Every square in $2\Omega A\rtimes_\Phi\Omega G$ is thus of the form
    \begin{center}

\tikzset{every picture/.style={line width=0.75pt}} 

\begin{tikzpicture}[x=0.75pt,y=0.75pt,yscale=-1,xscale=1]

\draw [color={rgb, 255:red, 208; green, 2; blue, 27 }  ,draw opacity=1 ]   (140,141.6) -- (190,141.6) ;
\draw [color={rgb, 255:red, 208; green, 2; blue, 27 }  ,draw opacity=1 ]   (140,191.6) -- (190,191.6) ;
\draw    (140,141.6) -- (140,191.6) ;
\draw    (190,141.6) -- (190,191.6) ;

\draw (192.36,160.04) node [anchor=north west][inner sep=0.75pt]  [font=\scriptsize]  {$g$};
\draw (129.36,160.04) node [anchor=north west][inner sep=0.75pt]  [font=\scriptsize]  {$g$};
\draw (160.61,160.9) node [anchor=north west][inner sep=0.75pt]  [font=\scriptsize]  {$a$};

\end{tikzpicture}

    \end{center}
    \noindent for some $g\in G$ and $a\in A$. By Observation \ref{obs:slidingdowntheorem} every such square can be subdivided as
    \begin{center}

\tikzset{every picture/.style={line width=0.75pt}} 

\begin{tikzpicture}[x=0.75pt,y=0.75pt,yscale=-1,xscale=1]

\draw [color={rgb, 255:red, 208; green, 2; blue, 27 }  ,draw opacity=1 ]   (160,161.6) -- (210,161.6) ;
\draw [color={rgb, 255:red, 208; green, 2; blue, 27 }  ,draw opacity=1 ]   (160,211.6) -- (210,211.6) ;
\draw    (160,161.6) -- (160,211.6) ;
\draw    (210,161.6) -- (210,211.6) ;
\draw [color={rgb, 255:red, 208; green, 2; blue, 27 }  ,draw opacity=1 ]   (160,261.6) -- (210,261.6) ;
\draw [color={rgb, 255:red, 208; green, 2; blue, 27 }  ,draw opacity=1 ]   (160,211.6) -- (160,261.6) ;
\draw [color={rgb, 255:red, 208; green, 2; blue, 27 }  ,draw opacity=1 ]   (210,211.6) -- (210,261.6) ;

\draw (212.36,180.04) node [anchor=north west][inner sep=0.75pt]  [font=\scriptsize]  {$g$};
\draw (149.36,180.04) node [anchor=north west][inner sep=0.75pt]  [font=\scriptsize]  {$g$};
\draw (180.61,228.9) node [anchor=north west][inner sep=0.75pt]  [font=\scriptsize]  {$a$};
\draw (177.61,182.9) node [anchor=north west][inner sep=0.75pt]  [font=\scriptsize]  {$U$};

\end{tikzpicture}

    \end{center}
    \noindent The horizontal composition of two such squares
    \begin{center}

\tikzset{every picture/.style={line width=0.75pt}} 

\begin{tikzpicture}[x=0.75pt,y=0.75pt,yscale=-1,xscale=1]

\draw [color={rgb, 255:red, 208; green, 2; blue, 27 }  ,draw opacity=1 ]   (180,179.6) -- (230,179.6) ;
\draw [color={rgb, 255:red, 208; green, 2; blue, 27 }  ,draw opacity=1 ]   (180,229.6) -- (230,229.6) ;
\draw    (180,179.6) -- (180,229.6) ;
\draw    (230,179.6) -- (230,229.6) ;
\draw [color={rgb, 255:red, 208; green, 2; blue, 27 }  ,draw opacity=1 ]   (180,279.6) -- (230,279.6) ;
\draw [color={rgb, 255:red, 208; green, 2; blue, 27 }  ,draw opacity=1 ]   (180,229.6) -- (180,279.6) ;
\draw [color={rgb, 255:red, 208; green, 2; blue, 27 }  ,draw opacity=1 ]   (230,229.6) -- (230,279.6) ;
\draw [color={rgb, 255:red, 208; green, 2; blue, 27 }  ,draw opacity=1 ]   (290,179.6) -- (340,179.6) ;
\draw [color={rgb, 255:red, 208; green, 2; blue, 27 }  ,draw opacity=1 ]   (290,229.6) -- (340,229.6) ;
\draw    (290,179.6) -- (290,229.6) ;
\draw    (340,179.6) -- (340,229.6) ;
\draw [color={rgb, 255:red, 208; green, 2; blue, 27 }  ,draw opacity=1 ]   (290,279.6) -- (340,279.6) ;
\draw [color={rgb, 255:red, 208; green, 2; blue, 27 }  ,draw opacity=1 ]   (290,229.6) -- (290,279.6) ;
\draw [color={rgb, 255:red, 208; green, 2; blue, 27 }  ,draw opacity=1 ]   (340,229.6) -- (340,279.6) ;

\draw (232.36,198.04) node [anchor=north west][inner sep=0.75pt]  [font=\scriptsize]  {$g$};
\draw (169.36,198.04) node [anchor=north west][inner sep=0.75pt]  [font=\scriptsize]  {$g$};
\draw (200.61,246.9) node [anchor=north west][inner sep=0.75pt]  [font=\scriptsize]  {$a$};
\draw (342.36,198.04) node [anchor=north west][inner sep=0.75pt]  [font=\scriptsize]  {$g$};
\draw (279.36,198.04) node [anchor=north west][inner sep=0.75pt]  [font=\scriptsize]  {$g$};
\draw (310.61,246.9) node [anchor=north west][inner sep=0.75pt]  [font=\scriptsize]  {$b$};
\draw (197.61,203.4) node [anchor=north west][inner sep=0.75pt]  [font=\scriptsize]  {$U$};
\draw (307.61,202.4) node [anchor=north west][inner sep=0.75pt]  [font=\scriptsize]  {$U$};

\end{tikzpicture}

    \end{center}
    \noindent is the square
    \begin{center}

\tikzset{every picture/.style={line width=0.75pt}} 

\begin{tikzpicture}[x=0.75pt,y=0.75pt,yscale=-1,xscale=1]

\draw [color={rgb, 255:red, 208; green, 2; blue, 27 }  ,draw opacity=1 ]   (269,112) -- (319,112) ;
\draw [color={rgb, 255:red, 208; green, 2; blue, 27 }  ,draw opacity=1 ]   (269,162) -- (319,162) ;
\draw    (269,112) -- (269,162) ;
\draw    (319,112) -- (319,162) ;
\draw [color={rgb, 255:red, 208; green, 2; blue, 27 }  ,draw opacity=1 ]   (269,212) -- (319,212) ;
\draw [color={rgb, 255:red, 208; green, 2; blue, 27 }  ,draw opacity=1 ]   (269,162) -- (269,212) ;
\draw [color={rgb, 255:red, 208; green, 2; blue, 27 }  ,draw opacity=1 ]   (319,162) -- (319,212) ;

\draw (321.36,130.44) node [anchor=north west][inner sep=0.75pt]  [font=\scriptsize]  {$g$};
\draw (258.36,130.44) node [anchor=north west][inner sep=0.75pt]  [font=\scriptsize]  {$g$};
\draw (287.61,179.3) node [anchor=north west][inner sep=0.75pt]  [font=\scriptsize]  {$ab$};
\draw (287.61,134.3) node [anchor=north west][inner sep=0.75pt]  [font=\scriptsize]  {$U$};

\end{tikzpicture}

    \end{center}
    \noindent Now, the vertical composition of two squares:
    \begin{center}

\tikzset{every picture/.style={line width=0.75pt}} 

\begin{tikzpicture}[x=0.75pt,y=0.75pt,yscale=-1,xscale=1]

\draw [color={rgb, 255:red, 208; green, 2; blue, 27 }  ,draw opacity=1 ]   (200,190) -- (250,190) ;
\draw [color={rgb, 255:red, 208; green, 2; blue, 27 }  ,draw opacity=1 ]   (200,240) -- (250,240) ;
\draw    (200,190) -- (200,240) ;
\draw    (250,190) -- (250,240) ;
\draw [color={rgb, 255:red, 208; green, 2; blue, 27 }  ,draw opacity=1 ]   (200,290) -- (250,290) ;
\draw [color={rgb, 255:red, 208; green, 2; blue, 27 }  ,draw opacity=1 ]   (200,240) -- (200,290) ;
\draw [color={rgb, 255:red, 208; green, 2; blue, 27 }  ,draw opacity=1 ]   (250,240) -- (250,290) ;
\draw [color={rgb, 255:red, 208; green, 2; blue, 27 }  ,draw opacity=1 ]   (310,190) -- (360,190) ;
\draw [color={rgb, 255:red, 208; green, 2; blue, 27 }  ,draw opacity=1 ]   (310,240) -- (360,240) ;
\draw    (310,190) -- (310,240) ;
\draw    (360,190) -- (360,240) ;
\draw [color={rgb, 255:red, 208; green, 2; blue, 27 }  ,draw opacity=1 ]   (310,290) -- (360,290) ;
\draw [color={rgb, 255:red, 208; green, 2; blue, 27 }  ,draw opacity=1 ]   (310,240) -- (310,290) ;
\draw [color={rgb, 255:red, 208; green, 2; blue, 27 }  ,draw opacity=1 ]   (360,240) -- (360,290) ;

\draw (252.36,208.44) node [anchor=north west][inner sep=0.75pt]  [font=\scriptsize]  {$g$};
\draw (189.36,208.44) node [anchor=north west][inner sep=0.75pt]  [font=\scriptsize]  {$g$};
\draw (220.61,257.3) node [anchor=north west][inner sep=0.75pt]  [font=\scriptsize]  {$a$};
\draw (362.36,208.44) node [anchor=north west][inner sep=0.75pt]  [font=\scriptsize]  {$h$};
\draw (299.36,208.44) node [anchor=north west][inner sep=0.75pt]  [font=\scriptsize]  {$h$};
\draw (330.61,257.3) node [anchor=north west][inner sep=0.75pt]  [font=\scriptsize]  {$b$};
\draw (219.36,212.44) node [anchor=north west][inner sep=0.75pt]  [font=\scriptsize]  {$U$};
\draw (328.56,211.64) node [anchor=north west][inner sep=0.75pt]  [font=\scriptsize]  {$U$};

\end{tikzpicture}

    \end{center}
    \noindent is the square
    \begin{center}

\tikzset{every picture/.style={line width=0.75pt}} 

\begin{tikzpicture}[x=0.75pt,y=0.75pt,yscale=-1,xscale=1]

\draw [color={rgb, 255:red, 208; green, 2; blue, 27 }  ,draw opacity=1 ]   (328,60.4) -- (378,60.4) ;
\draw [color={rgb, 255:red, 208; green, 2; blue, 27 }  ,draw opacity=1 ]   (328,110.4) -- (378,110.4) ;
\draw    (328,60.4) -- (328,110.4) ;
\draw    (378,60.4) -- (378,110.4) ;
\draw [color={rgb, 255:red, 208; green, 2; blue, 27 }  ,draw opacity=1 ]   (328,110.4) -- (328,160.4) ;
\draw [color={rgb, 255:red, 208; green, 2; blue, 27 }  ,draw opacity=1 ]   (378,110.4) -- (378,160.4) ;
\draw [color={rgb, 255:red, 208; green, 2; blue, 27 }  ,draw opacity=1 ]   (328,160.4) -- (378,160.4) ;
\draw [color={rgb, 255:red, 208; green, 2; blue, 27 }  ,draw opacity=1 ]   (328,210) -- (378,210) ;
\draw    (328,160) -- (328,210) ;
\draw    (378,160) -- (378,210) ;
\draw [color={rgb, 255:red, 208; green, 2; blue, 27 }  ,draw opacity=1 ]   (328,260) -- (378,260) ;
\draw [color={rgb, 255:red, 208; green, 2; blue, 27 }  ,draw opacity=1 ]   (328,210) -- (328,260) ;
\draw [color={rgb, 255:red, 208; green, 2; blue, 27 }  ,draw opacity=1 ]   (378,210) -- (378,260) ;

\draw (380.36,78.84) node [anchor=north west][inner sep=0.75pt]  [font=\scriptsize]  {$g$};
\draw (317.36,78.84) node [anchor=north west][inner sep=0.75pt]  [font=\scriptsize]  {$g$};
\draw (348.61,127.7) node [anchor=north west][inner sep=0.75pt]  [font=\scriptsize]  {$a$};
\draw (380.36,178.44) node [anchor=north west][inner sep=0.75pt]  [font=\scriptsize]  {$h$};
\draw (317.36,178.44) node [anchor=north west][inner sep=0.75pt]  [font=\scriptsize]  {$h$};
\draw (348.61,227.3) node [anchor=north west][inner sep=0.75pt]  [font=\scriptsize]  {$b$};
\draw (346.21,82.5) node [anchor=north west][inner sep=0.75pt]  [font=\scriptsize]  {$U$};
\draw (345.81,180.8) node [anchor=north west][inner sep=0.75pt]  [font=\scriptsize]  {$U$};

\end{tikzpicture}

    \end{center}
    \noindent Which is equal to the square:
    \begin{center}

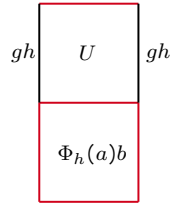
\tikzset{every picture/.style={line width=0.75pt}} 

\begin{tikzpicture}[x=0.75pt,y=0.75pt,yscale=-1,xscale=1]

\draw [color={rgb, 255:red, 208; green, 2; blue, 27 }  ,draw opacity=1 ]   (290,89) -- (340,89) ;
\draw [color={rgb, 255:red, 208; green, 2; blue, 27 }  ,draw opacity=1 ]   (290,139) -- (340,139) ;
\draw    (290,89) -- (290,139) ;
\draw    (340,89) -- (340,139) ;
\draw [color={rgb, 255:red, 208; green, 2; blue, 27 }  ,draw opacity=1 ]   (290,189) -- (340,189) ;
\draw [color={rgb, 255:red, 208; green, 2; blue, 27 }  ,draw opacity=1 ]   (290,139) -- (290,189) ;
\draw [color={rgb, 255:red, 208; green, 2; blue, 27 }  ,draw opacity=1 ]   (340,139) -- (340,189) ;

\draw (342.36,107.44) node [anchor=north west][inner sep=0.75pt]  [font=\scriptsize]  {$gh$};
\draw (274.36,107.44) node [anchor=north west][inner sep=0.75pt]  [font=\scriptsize]  {$gh$};
\draw (297.61,158.3) node [anchor=north west][inner sep=0.75pt]  [font=\scriptsize]  {$\Phi _{h}( a) b$};
\draw (308.36,108.44) node [anchor=north west][inner sep=0.75pt]  [font=\scriptsize]  {$U$};

\end{tikzpicture}

    \end{center}
    \noindent The category of morphisms $2\Omega A\rtimes_\Phi \Omega G_1$ of $2\Omega A\rtimes_\Phi\Omega G$ can thus be identified with the delooping groupoid $\Omega(A\rtimes_\Phi G)$ of the semidirect product $A\rtimes_\Phi G$. An analogous argument identifies the category of mosphisms of internalizations of $(\Omega G,2\Omega A)$ of the form $2\Omega A\rtimes_\Phi \Omega G$ with the delooping groupoid of $A\rtimes_\Phi G^{op}$ for a $\pi_2$-opindexing $\Phi\in\pi_2\textbf{opInd}_{(\Omega G,2\Omega A)}$. 
\end{ex}
\begin{ex}\label{ex:tensorcatsdoublecats}
    Let $G$ be a group. Let $\mathbb{K}$ be a field, and let $C$ be a strict tensor category over $\mathbb{K}$. Let $\Phi\in \pi_2\textbf{Ind}_{(\Omega G,\Omega C)}$. Example \ref{ex:tensorcats} says that we can identify $\Phi$ with a character of $G$ on $\mathbb{K}$. Squares in $\Omega G\rtimes_\Phi \Omega C$ are of the form:
    \begin{center}

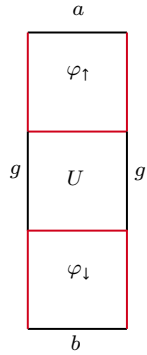
\tikzset{every picture/.style={line width=0.75pt}} 

\begin{tikzpicture}[x=0.75pt,y=0.75pt,yscale=-1,xscale=1]

\draw [color={rgb, 255:red, 0; green, 0; blue, 0 }  ,draw opacity=1 ]   (280,49.4) -- (330,49.4) ;
\draw [color={rgb, 255:red, 208; green, 2; blue, 27 }  ,draw opacity=1 ]   (280,99.4) -- (330,99.4) ;
\draw [color={rgb, 255:red, 208; green, 2; blue, 27 }  ,draw opacity=1 ]   (280,49.4) -- (280,99.4) ;
\draw [color={rgb, 255:red, 208; green, 2; blue, 27 }  ,draw opacity=1 ]   (330,49.4) -- (330,99.4) ;
\draw [color={rgb, 255:red, 0; green, 0; blue, 0 }  ,draw opacity=1 ]   (280,99.4) -- (280,149.4) ;
\draw [color={rgb, 255:red, 0; green, 0; blue, 0 }  ,draw opacity=1 ]   (330,99.4) -- (330,149.4) ;
\draw [color={rgb, 255:red, 208; green, 2; blue, 27 }  ,draw opacity=1 ]   (280,149.4) -- (330,149.4) ;
\draw [color={rgb, 255:red, 0; green, 0; blue, 0 }  ,draw opacity=1 ]   (280,199) -- (330,199) ;
\draw [color={rgb, 255:red, 208; green, 2; blue, 27 }  ,draw opacity=1 ]   (280,149) -- (280,199) ;
\draw [color={rgb, 255:red, 208; green, 2; blue, 27 }  ,draw opacity=1 ]   (330,149) -- (330,199) ;

\draw (332.36,115.84) node [anchor=north west][inner sep=0.75pt]  [font=\scriptsize]  {$g$};
\draw (269.36,114.84) node [anchor=north west][inner sep=0.75pt]  [font=\scriptsize]  {$g$};
\draw (301.01,34.3) node [anchor=north west][inner sep=0.75pt]  [font=\scriptsize]  {$a$};
\draw (298.36,165.44) node [anchor=north west][inner sep=0.75pt]  [font=\scriptsize]  {$\varphi _{\downarrow }$};
\draw (298.21,117.5) node [anchor=north west][inner sep=0.75pt]  [font=\scriptsize]  {$U$};
\draw (300.21,200.6) node [anchor=north west][inner sep=0.75pt]  [font=\scriptsize]  {$b$};
\draw (297.96,64.64) node [anchor=north west][inner sep=0.75pt]  [font=\scriptsize]  {$\varphi _{\uparrow }$};

\end{tikzpicture}

    \end{center}
    \noindent where $a,b$ are objects in $C$, $\varphi:a\to\mathbb{K}$ and $\psi:\mathbb{K}\to b$ are a costate on $a$ and a state in $b$ respectively, and where $g\in G$. The composition of two such squares is:
    \begin{center}

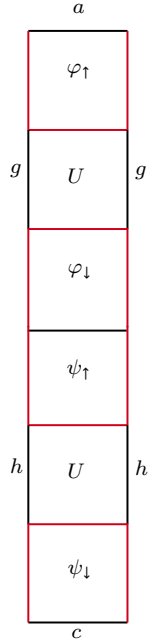
\tikzset{every picture/.style={line width=0.75pt}} 

\begin{tikzpicture}[x=0.75pt,y=0.75pt,yscale=-1,xscale=1]

\draw [color={rgb, 255:red, 0; green, 0; blue, 0 }  ,draw opacity=1 ]   (300,69.4) -- (350,69.4) ;
\draw [color={rgb, 255:red, 208; green, 2; blue, 27 }  ,draw opacity=1 ]   (300,119.4) -- (350,119.4) ;
\draw [color={rgb, 255:red, 208; green, 2; blue, 27 }  ,draw opacity=1 ]   (300,69.4) -- (300,119.4) ;
\draw [color={rgb, 255:red, 208; green, 2; blue, 27 }  ,draw opacity=1 ]   (350,69.4) -- (350,119.4) ;
\draw [color={rgb, 255:red, 0; green, 0; blue, 0 }  ,draw opacity=1 ]   (300,119.4) -- (300,169.4) ;
\draw [color={rgb, 255:red, 0; green, 0; blue, 0 }  ,draw opacity=1 ]   (350,119.4) -- (350,169.4) ;
\draw [color={rgb, 255:red, 208; green, 2; blue, 27 }  ,draw opacity=1 ]   (300,169.4) -- (350,169.4) ;
\draw [color={rgb, 255:red, 0; green, 0; blue, 0 }  ,draw opacity=1 ]   (299.67,220.67) -- (349.67,220.67) ;
\draw [color={rgb, 255:red, 208; green, 2; blue, 27 }  ,draw opacity=1 ]   (300,169) -- (300,219) ;
\draw [color={rgb, 255:red, 208; green, 2; blue, 27 }  ,draw opacity=1 ]   (350,169) -- (350,219) ;
\draw [color={rgb, 255:red, 208; green, 2; blue, 27 }  ,draw opacity=1 ]   (300,268.4) -- (350,268.4) ;
\draw [color={rgb, 255:red, 208; green, 2; blue, 27 }  ,draw opacity=1 ]   (300,218.4) -- (300,268.4) ;
\draw [color={rgb, 255:red, 208; green, 2; blue, 27 }  ,draw opacity=1 ]   (350,218.4) -- (350,268.4) ;
\draw [color={rgb, 255:red, 0; green, 0; blue, 0 }  ,draw opacity=1 ]   (300,268.4) -- (300,318.4) ;
\draw [color={rgb, 255:red, 0; green, 0; blue, 0 }  ,draw opacity=1 ]   (350,268.4) -- (350,318.4) ;
\draw [color={rgb, 255:red, 208; green, 2; blue, 27 }  ,draw opacity=1 ]   (300,318.4) -- (350,318.4) ;
\draw [color={rgb, 255:red, 0; green, 0; blue, 0 }  ,draw opacity=1 ]   (300,368) -- (350,368) ;
\draw [color={rgb, 255:red, 208; green, 2; blue, 27 }  ,draw opacity=1 ]   (300,318) -- (300,368) ;
\draw [color={rgb, 255:red, 208; green, 2; blue, 27 }  ,draw opacity=1 ]   (350,318) -- (350,368) ;

\draw (352.36,135.84) node [anchor=north west][inner sep=0.75pt]  [font=\scriptsize]  {$g$};
\draw (289.36,134.84) node [anchor=north west][inner sep=0.75pt]  [font=\scriptsize]  {$g$};
\draw (321.01,54.3) node [anchor=north west][inner sep=0.75pt]  [font=\scriptsize]  {$a$};
\draw (318.36,185.44) node [anchor=north west][inner sep=0.75pt]  [font=\scriptsize]  {$\varphi _{\downarrow }$};
\draw (318.21,137.5) node [anchor=north west][inner sep=0.75pt]  [font=\scriptsize]  {$U$};
\draw (317.96,84.64) node [anchor=north west][inner sep=0.75pt]  [font=\scriptsize]  {$\varphi _{\uparrow }$};
\draw (352.36,284.84) node [anchor=north west][inner sep=0.75pt]  [font=\scriptsize]  {$h$};
\draw (289.36,283.84) node [anchor=north west][inner sep=0.75pt]  [font=\scriptsize]  {$h$};
\draw (318.36,334.44) node [anchor=north west][inner sep=0.75pt]  [font=\scriptsize]  {$\psi _{\downarrow }$};
\draw (318.21,286.5) node [anchor=north west][inner sep=0.75pt]  [font=\scriptsize]  {$U$};
\draw (320.21,369.6) node [anchor=north west][inner sep=0.75pt]  [font=\scriptsize]  {$c$};
\draw (317.96,233.64) node [anchor=north west][inner sep=0.75pt]  [font=\scriptsize]  {$\psi _{\uparrow }$};

\end{tikzpicture}

    \end{center}
    \noindent which is equal to
    \begin{center}

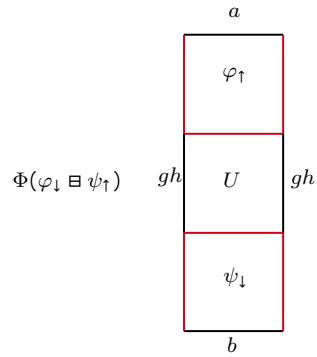
\tikzset{every picture/.style={line width=0.75pt}} 

\begin{tikzpicture}[x=0.75pt,y=0.75pt,yscale=-1,xscale=1]

\draw [color={rgb, 255:red, 0; green, 0; blue, 0 }  ,draw opacity=1 ]   (300,69.4) -- (350,69.4) ;
\draw [color={rgb, 255:red, 208; green, 2; blue, 27 }  ,draw opacity=1 ]   (300,119.4) -- (350,119.4) ;
\draw [color={rgb, 255:red, 208; green, 2; blue, 27 }  ,draw opacity=1 ]   (300,69.4) -- (300,119.4) ;
\draw [color={rgb, 255:red, 208; green, 2; blue, 27 }  ,draw opacity=1 ]   (350,69.4) -- (350,119.4) ;
\draw [color={rgb, 255:red, 0; green, 0; blue, 0 }  ,draw opacity=1 ]   (300,119.4) -- (300,169.4) ;
\draw [color={rgb, 255:red, 0; green, 0; blue, 0 }  ,draw opacity=1 ]   (350,119.4) -- (350,169.4) ;
\draw [color={rgb, 255:red, 208; green, 2; blue, 27 }  ,draw opacity=1 ]   (300,169.4) -- (350,169.4) ;
\draw [color={rgb, 255:red, 0; green, 0; blue, 0 }  ,draw opacity=1 ]   (300,219) -- (350,219) ;
\draw [color={rgb, 255:red, 208; green, 2; blue, 27 }  ,draw opacity=1 ]   (300,169) -- (300,219) ;
\draw [color={rgb, 255:red, 208; green, 2; blue, 27 }  ,draw opacity=1 ]   (350,169) -- (350,219) ;

\draw (352.36,135.84) node [anchor=north west][inner sep=0.75pt]  [font=\scriptsize]  {$gh$};
\draw (285.36,134.84) node [anchor=north west][inner sep=0.75pt]  [font=\scriptsize]  {$gh$};
\draw (321.01,54.3) node [anchor=north west][inner sep=0.75pt]  [font=\scriptsize]  {$a$};
\draw (318.36,185.44) node [anchor=north west][inner sep=0.75pt]  [font=\scriptsize]  {$\psi _{\downarrow }$};
\draw (318.21,137.5) node [anchor=north west][inner sep=0.75pt]  [font=\scriptsize]  {$U$};
\draw (320.21,220.6) node [anchor=north west][inner sep=0.75pt]  [font=\scriptsize]  {$b$};
\draw (317.96,84.64) node [anchor=north west][inner sep=0.75pt]  [font=\scriptsize]  {$\varphi _{\uparrow }$};
\draw (212,136.4) node [anchor=north west][inner sep=0.75pt]  [font=\scriptsize]  {$\Phi ( \varphi _{\downarrow } \boxminus \psi _{\uparrow })$};

\end{tikzpicture}

    \end{center}
    \noindent where the above notation represents multiplication by the scalar $\Phi(\varphi_\downarrow\boxminus\psi_\uparrow)$.
\end{ex}
\begin{obs}\label{obs:bracket}
    Example \ref{ex:tensorcatsdoublecats} suggests the following notational device: Given a decorated 2-category $\decB$ and $\Phi\in \EndCat$, we denote the square
    \begin{center}

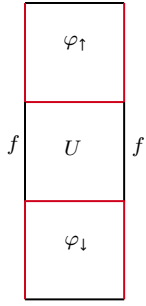
\tikzset{every picture/.style={line width=0.75pt}} 

\begin{tikzpicture}[x=0.75pt,y=0.75pt,yscale=-1,xscale=1]

\draw [color={rgb, 255:red, 0; green, 0; blue, 0 }  ,draw opacity=1 ]   (300,69.4) -- (350,69.4) ;
\draw [color={rgb, 255:red, 208; green, 2; blue, 27 }  ,draw opacity=1 ]   (300,119.4) -- (350,119.4) ;
\draw [color={rgb, 255:red, 208; green, 2; blue, 27 }  ,draw opacity=1 ]   (300,69.4) -- (300,119.4) ;
\draw [color={rgb, 255:red, 208; green, 2; blue, 27 }  ,draw opacity=1 ]   (350,69.4) -- (350,119.4) ;
\draw [color={rgb, 255:red, 0; green, 0; blue, 0 }  ,draw opacity=1 ]   (300,119.4) -- (300,169.4) ;
\draw [color={rgb, 255:red, 0; green, 0; blue, 0 }  ,draw opacity=1 ]   (350,119.4) -- (350,169.4) ;
\draw [color={rgb, 255:red, 208; green, 2; blue, 27 }  ,draw opacity=1 ]   (300,169.4) -- (350,169.4) ;
\draw [color={rgb, 255:red, 0; green, 0; blue, 0 }  ,draw opacity=1 ]   (300,219) -- (350,219) ;
\draw [color={rgb, 255:red, 208; green, 2; blue, 27 }  ,draw opacity=1 ]   (300,169) -- (300,219) ;
\draw [color={rgb, 255:red, 208; green, 2; blue, 27 }  ,draw opacity=1 ]   (350,169) -- (350,219) ;

\draw (352.36,135.84) node [anchor=north west][inner sep=0.75pt]  [font=\scriptsize]  {$f$};
\draw (289.36,134.84) node [anchor=north west][inner sep=0.75pt]  [font=\scriptsize]  {$f$};
\draw (318.36,185.44) node [anchor=north west][inner sep=0.75pt]  [font=\scriptsize]  {$\varphi _{\downarrow }$};
\draw (318.21,137.5) node [anchor=north west][inner sep=0.75pt]  [font=\scriptsize]  {$U$};
\draw (317.96,84.64) node [anchor=north west][inner sep=0.75pt]  [font=\scriptsize]  {$\varphi _{\uparrow }$};

\end{tikzpicture}

    \end{center}
    \noindent in $\crossb$ as $|\varphi_\downarrow\rangle f\langle \varphi_\uparrow|$, in analogy with the Bra-ket notation. In that case the horizontal composition of two horizontally compatible squares $|\varphi_\downarrow\rangle f\langle \varphi_\uparrow|$ and $|\psi_\downarrow\rangle f\langle \psi_\uparrow|$ is equal to $|\varphi_\downarrow\boxvert\psi_\downarrow\rangle f\langle \varphi_\uparrow\boxvert\psi_\uparrow|$. We write the vertical composite of two vertically compatible squares $|\varphi_\downarrow\rangle f\langle \varphi_\uparrow|$ and $|\psi_\downarrow\rangle g\langle \psi_\uparrow|$ as the square $\langle\psi_\uparrow|\varphi_\downarrow\rangle|\psi_\downarrow\rangle gf\langle \varphi_\uparrow|$.
\end{obs}

\begin{ex}\label{ex:spansdoublecats}
    Let $C$ be a category with enough pullbacks. Assume a choice of pullback for every span in $C$ has been made, so as to make $SpanC$ into a 2-category. The pair $(C,SpanC)$ is a decorated 2-category. Let $\Phi\in\pi_2\textbf{Ind}_{(C,SpanC)}$ be the unique $\pi_2$-indexing of $(C,SpanC)$. Every square in $SpanC$ of the form
    \begin{center}

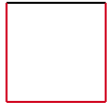
\tikzset{every picture/.style={line width=0.75pt}} 

\begin{tikzpicture}[x=0.75pt,y=0.75pt,yscale=-1,xscale=1]

\draw    (241,140.6) -- (291,140.6) ;
\draw [color={rgb, 255:red, 208; green, 2; blue, 27 }  ,draw opacity=1 ]   (241,190.6) -- (291,190.6) ;
\draw [color={rgb, 255:red, 208; green, 2; blue, 27 }  ,draw opacity=1 ]   (241,140.6) -- (241,190.6) ;
\draw [color={rgb, 255:red, 208; green, 2; blue, 27 }  ,draw opacity=1 ]   (291,140.6) -- (291,190.6) ;

\end{tikzpicture}

    \end{center}
    \noindent is a span of the form
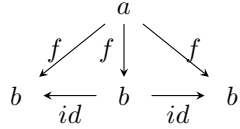
\begin{center}
    \begin{tikzpicture}
  \matrix (m) [matrix of math nodes,row sep=2em,column sep=2em,minimum width=2em]
  {
     &a& \\
     b&b&b \\};
  \path[-stealth]
    (m-1-2) edge node [left]{$f$}(m-2-1)
            edge node [left]{$f$}(m-2-2)
            edge node [right]{$f$}(m-2-3)
    (m-2-2) edge node [below] {$id$}(m-2-1)
            edge node [below] {$id$}(m-2-3);
\end{tikzpicture}
\end{center}
\noindent for a morphism $f$ in $C$. A square in $SpanC$ of the form
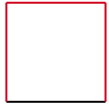
\begin{center}

\tikzset{every picture/.style={line width=0.75pt}} 

\begin{tikzpicture}[x=0.75pt,y=0.75pt,yscale=-1,xscale=1]

\draw [color={rgb, 255:red, 208; green, 2; blue, 27 }  ,draw opacity=1 ]   (261,160.6) -- (311,160.6) ;
\draw [color={rgb, 255:red, 0; green, 0; blue, 0 }  ,draw opacity=1 ]   (261,210.6) -- (311,210.6) ;
\draw [color={rgb, 255:red, 208; green, 2; blue, 27 }  ,draw opacity=1 ]   (261,160.6) -- (261,210.6) ;
\draw [color={rgb, 255:red, 208; green, 2; blue, 27 }  ,draw opacity=1 ]   (311,160.6) -- (311,210.6) ;

\end{tikzpicture}

\end{center}
\noindent is a span of the form
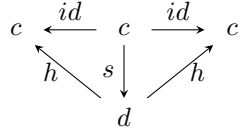
\begin{center}
    \begin{tikzpicture}
  \matrix (m) [matrix of math nodes,row sep=2em,column sep=2em,minimum width=2em]
  {
     c&c&c \\
      &d& \\};
  \path[-stealth]
    (m-2-2) edge node [left]{$h$}(m-1-1)
    (m-1-2) edge node [left]{$s$}(m-2-2)
    (m-2-2) edge node [right]{$h$}(m-1-3)
    (m-1-2) edge node [above]{$id$}(m-1-1)
            edge node [above]{$id$}(m-1-3);
\end{tikzpicture}
\end{center}
\noindent for some morphisms $h, s$ in $C$. We suggestively represent the unit square $U_g$ of a morphism $g$ in $C$ as
\begin{center}
    \begin{tikzpicture}
  \matrix (m) [matrix of math nodes,row sep=2em,column sep=2em,minimum width=2em]
  {
     b&b&b \\
     c&c&c \\};
  \path[-stealth]
    (m-1-2) edge node [above]{$id$}(m-1-1)
            edge node [left] {$g$}(m-2-2)
            edge node [above]{$id$}(m-1-3)
    (m-1-1) edge node [left]{$g$}(m-2-1)
    (m-1-3) edge node [right]{$g$}(m-2-3)
    (m-2-2) edge node [below]{$id$}(m-2-1)
    (m-2-2) edge node [below]{$id$}(m-2-3);
\end{tikzpicture}
\end{center}
With this notation, squares in $SpanC\rtimes_\Phi C$ are all of the form
\begin{center}
    \begin{tikzpicture}
  \matrix (m) [matrix of math nodes,row sep=2em,column sep=2em,minimum width=2em]
  {
      &a& \\
     b&b&b \\
     c&c&c&\\
      &d&\\};
  \path[-stealth]
    (m-1-2) edge node [left]{$f$}(m-2-1)
            edge node [left]{$f$}(m-2-2)
            edge node [right]{$f$}(m-2-3)
    (m-2-2) edge node [below] {$id$}(m-2-1)
            edge node [below] {$id$}(m-2-3)

    (m-2-1) edge node [left]{$g$}(m-3-1)
    (m-2-2) edge node [left]{$g$}(m-3-2)
    (m-2-3) edge node [right]{$g$}(m-3-3)
    
    (m-3-2) edge node [below]{$id$}(m-3-1)
    (m-3-2) edge node [below]{$id$}(m-3-3)

    (m-4-2) edge node [left]{$h$}(m-3-1)
    (m-3-2) edge node [left]{$s$}(m-4-2)
    (m-4-2) edge node [right]{$h$}(m-3-3);
\end{tikzpicture}
\end{center}

\end{ex}
\noindent The following is an example of a globularly generated double category $C$ of length 1, such that not every square in $C$ admits a canonical decomposition and thus is not of the form $\crossb$ for a $\pi_2$-indexing or a $\pi_2$-opindexing $\Phi$.
\begin{ex}\label{ex:notcrossedprod}
Let $\stB$ be the category \textbf{2} freely generated by arrows:

\begin{center}
\begin{tikzpicture}
\matrix(m)[matrix of math nodes, row sep=3em, column sep=3em,text height=1.5ex, text depth=0.25ex]
{0\\
1\\
2\\};
\path[->,font=\scriptsize,>=angle 90]
(m-1-1) edge node [left]{$\alpha$} (m-2-1)
(m-2-1) edge node [left]{$\beta$}(m-3-1)

;
\end{tikzpicture}
\end{center}

\noindent Let $B$ be the 2-category with 0-cells 0,1,2, with only horizontal identity cells, and such that $\pi_2(0)=\mthz/2\mthz,\pi_2(1)=\mthz/2\mthz$, and $\pi_2(2)=1$. The pair $\decB$ is a decorated 2-category. Let $Q_{\decB}$ be the free globularly generated double category generated by $\decB$, see \cite{Orendain2}. $Q_{\decB}$ is a globularly generated internalization of $\decB$, see \cite[Proposition 3.2]{Orendain2}. Moreover, it is easily seen that the horizontal composition of any two squares of length 1 in $Q_{\decB}$ is again of length 1, and thus $\ell Q_{\decB}=1$. The square:

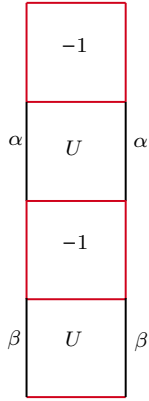
\begin{center}

\tikzset{every picture/.style={line width=0.75pt}} 

\begin{tikzpicture}[x=0.75pt,y=0.75pt,yscale=-1,xscale=1]

\draw [color={rgb, 255:red, 208; green, 2; blue, 27 }  ,draw opacity=1 ]   (320,89.4) -- (370,89.4) ;
\draw [color={rgb, 255:red, 208; green, 2; blue, 27 }  ,draw opacity=1 ]   (320,139.4) -- (370,139.4) ;
\draw [color={rgb, 255:red, 208; green, 2; blue, 27 }  ,draw opacity=1 ]   (320,89.4) -- (320,139.4) ;
\draw [color={rgb, 255:red, 208; green, 2; blue, 27 }  ,draw opacity=1 ]   (370,89.4) -- (370,139.4) ;
\draw [color={rgb, 255:red, 0; green, 0; blue, 0 }  ,draw opacity=1 ]   (320,139.4) -- (320,189.4) ;
\draw [color={rgb, 255:red, 0; green, 0; blue, 0 }  ,draw opacity=1 ]   (370,139.4) -- (370,189.4) ;
\draw [color={rgb, 255:red, 208; green, 2; blue, 27 }  ,draw opacity=1 ]   (320,189.4) -- (370,189.4) ;
\draw [color={rgb, 255:red, 208; green, 2; blue, 27 }  ,draw opacity=1 ]   (320,239) -- (370,239) ;
\draw [color={rgb, 255:red, 208; green, 2; blue, 27 }  ,draw opacity=1 ]   (320,189) -- (320,239) ;
\draw [color={rgb, 255:red, 208; green, 2; blue, 27 }  ,draw opacity=1 ]   (370,189) -- (370,239) ;
\draw [color={rgb, 255:red, 208; green, 2; blue, 27 }  ,draw opacity=1 ]   (320,288.4) -- (370,288.4) ;
\draw [color={rgb, 255:red, 0; green, 0; blue, 0 }  ,draw opacity=1 ]   (320,238.4) -- (320,288.4) ;
\draw [color={rgb, 255:red, 0; green, 0; blue, 0 }  ,draw opacity=1 ]   (370,238.4) -- (370,288.4) ;

\draw (372.36,155.84) node [anchor=north west][inner sep=0.75pt]  [font=\scriptsize]  {$\alpha $};
\draw (309.36,154.84) node [anchor=north west][inner sep=0.75pt]  [font=\scriptsize]  {$\alpha $};
\draw (336.36,205.44) node [anchor=north west][inner sep=0.75pt]  [font=\scriptsize]  {$-1$};
\draw (338.21,157.5) node [anchor=north west][inner sep=0.75pt]  [font=\scriptsize]  {$U$};
\draw (335.96,105.64) node [anchor=north west][inner sep=0.75pt]  [font=\scriptsize]  {$-1$};
\draw (337.96,253.64) node [anchor=north west][inner sep=0.75pt]  [font=\scriptsize]  {$U$};
\draw (308.96,253.04) node [anchor=north west][inner sep=0.75pt]  [font=\scriptsize]  {$\beta $};
\draw (372.96,254.4) node [anchor=north west][inner sep=0.75pt]  [font=\scriptsize]  {$\beta $};

\end{tikzpicture}

\end{center}

\noindent in $Q_{\decB}$ cannot be written as the vertical composition of less than four globular and horizontal identity squares of $Q_{\decB}$ and thus does not admit canonical decompositions. We conclude that $Q_{\decB}$ is not of the form $\crossb$ for a $\Phi\in\EndCat$ nor $\Phi\in \opIndCat$.

\end{ex}

\section{Outlook}\label{s:outlook}

\noindent In this last section we present open problems and lines of study related to the material presented in this paper.
\begin{enumerate}
    \item \textbf{Coherence}: One of the obvious shortcomings of the construction presented in Theorem \ref{thm:main} is that it only considers decorated 2-categories. A version of this construction for decorated bicategories would, in principle, be analogous to its 2-categorical counterpart, but unit isomorphisms, and associators must come into play. We consider that in the context of this paper, this extra structure would obscure the arguments in the proof of Theorem \ref{thm:main} and were thus omitted. A version of our results for decorated bicategories will appear in future work. Another point where the strictness of our construction might be weakened is in the fact that $\pi_2$-indexings were taken to be strict functors. One can think of a version of the construction $\crossb$ for pseudofunctors, or lax/oplax functors $\Phi:\stB\to\Cat$. These conditions will again be considered in future work.  
    \item \textbf{Functoriality}: We have not addressed any questions of functoriality of the construction appearing in the proof of Theorem \ref{thm:main}. There are different ways of providing collections of $\pi_2$-indexings with the structure of a category. The problem of choosing the best possible such structure on $\EndCat$ making the construction $\crossb$ functorial, in the appropriate sense, will be addressed in future work. Functoriality is related to the problem of comparison, i.e. the problem of deciding when two double categories of the form $\crossb$ and $B\rtimes_\Psi B^\ast$ are double equivalent, in terms of the $\pi_2$-indexings $\Phi$ and $\Psi$. We conjecture that a categorical structure on $\EndCat$ can be chosen so as to make the construction $\crossb$ into an embedding of $\EndCat$ into $\gcat$, and thus the two double categories $\crossb$ and $B\rtimes_\Psi B^\ast$ would be double equivalent if and only if $\Phi$ and $\Psi$ are equivalent in the appropriate sense. 
    \item \textbf{Frameability}: One of the main questions we are interested in is when a decorated bicategory $\decB$ admits an internalization that is a framed bicategory. In the particular context of this paper we are interested in the question of when a double category of the form $\crossb$ is the globularly generated piece $\gamma C$ of a framed bicategory. This question will be addressed in future work.
    \end{enumerate}

    \section{Author's declarations}
\noindent We declare that there is no competing interests for this manuscript. We declare that there is no Availability of Data and Materials for this manuscript. Juan Orendain and Jos\'e Rub\'en Maldonado-Herrera contributed to the manuscript equally. Juan Orendain did not receive financial support. Jose Ruben Maldonado-Herrera was financially supported by CONAHCYT program 001458.

\noindent \textbf{Acknowledgements:} The authors would like to thank the anonymous referee, whose valuable comments greatly improved the quality of the paper.

\bibliographystyle{unsrt}
\bibliography{biblio}

\end{document}